\documentclass[11pt]{amsart}
\usepackage{geometry}
\usepackage{amscd,amssymb,verbatim,xcolor}
\usepackage{longtable, multirow}
\usepackage{soul,cancel,tikz}
\usepackage{pgf,tikz,pgfplots}
\pgfplotsset{compat=1.15}
\usepackage{mathrsfs}
\usetikzlibrary{arrows.meta,bending}
\date{\today}

\newtheorem{theorem}{Theorem}[section]

\newtheorem{corollary}[theorem]{Corollary}

\newtheorem{definition}[theorem]{Definition}
\newtheorem{example}[theorem]{Example}
\newtheorem{lemma}[theorem]{Lemma}
\newtheorem{proposition}[theorem]{Proposition}

\begin{document}
\title[Dirac series of $U(n,2)$]{Dirac series of $U(n,2)$}
\author{Lin Wei} \author{Chengyu Du} \author{Hongfeng Zhang} \author{Xuanchen Zhao}

\address{Department of Mathematics, Soochow University}
\email{1678728903@qq.com}

\address{Department of Mathematics, Soochow University}
\email{dcystory@163.com}

\address{Department of Mathematics, the University of Hong Kong, Pokfulam, Hong Kong, China}
\email{zhanghongf@pku.edu.cn}

\address{School of Science and Engineering, the Chinese University of Hong Kong, Shenzhen}
\email{116010313@link.cuhk.edu.cn}

\begin{abstract}
We determine the Dirac series of $U(n,2)$ and give related information about Dirac series: the spin lowest $K$-types are multiplicity free, and the Dirac cohomology are consistent with the Dirac index.
\end{abstract}

\maketitle

\section{Introduction}

The Dirac operator was introduced by Dirac as a square root of the Laplacian operator. Later, Parthasarathy \cite{Parthasarathy1972} defined a Dirac operator over $G/K$ in a similar way, and used it to understand the discrete series and gave an inequality, known as Dirac inequality, as a necessary condition for a representation to be unitarizable. To understand the unitary dual of a real reductive Lie group better, Vogan formulated the notion of Dirac cohomology in 1997 in MIT seminar and conjectured that whenever nonzero, the Dirac cohomology should reveal the infinitesimal character of the original module, which can be regarded as a refinement of Dirac inequality. This conjecture was confirmed by Huang and Pand\v zi\'c \cite{HP2002} in 2002;
see Theorem \ref{thm.HP.condition} below. Since then, the Dirac cohomology became a new invariant for unitary representations of real reductive Lie groups, and classifying all the irreducible unitary representations (up to equivalence) with nonvanishing Dirac cohomology, which are called the \emph{Dirac series}, became an interesting problem that remained open.

Assume that $G$ and $K$ are of equal rank, given any $(\mathfrak{g},K)$-module, from Dirac cohomology (could be zero), one can get the Dirac index, which is defined as the difference between ``odd" and ``even" part of the Dirac cohomology. The Dirac index was found to be easier to understood, which behaved well under coherent continuation. However, the cancellation between the even part and the odd part of Dirac cohomology may happen, see \cite{DDW23} about $E_{7(7)}$. Assume that $G$ is of Hermitian symmetric type, the Dirac cohomology of its unitary representations can be identified with the $\mathfrak{p}^+$-homology or $\mathfrak{p}^-$-cohomology up to a twist, see \cite{HPR2006}. As an application, the Dirac cohomology of the highest weight modules, which may be not unitary, could be computed via the Kazhdan-Lusztig polynomials, see \cite{HX12}. We will work out all Dirac series of $U(n,2)$, which is a group of Hermitian symmetric type. One might hope a well understanding of Dirac cohomology could be worthwhile for understanding the Shimura variety.

Recently, Wong \cite{Wong24} made important progress on the unitary dual of the group $U(p,q)$. Soon, Wong-Zhang  \cite{WongZhang24} classified the unitary dual for $U(n,2)$. In this paper, we will describe all the Dirac series for $U(n,2)$. The paper is formulated as follows. In section 2, we explain the irreducible representations of $U(n,2)$ using  combinatorial methods. In section 3, we show some fundamental results for the Dirac series. In section 4, we list all fully supported representations with nonzero Dirac cohomology as building blocks to construct all the Dirac series. In section 5, we compute the spin lowest $K$-types. From the results of Dirac index in \cite[Section 4]{DongWong21}, we verify the non-cancellation when passing Dirac cohomology to Dirac index.

\section{Irreducible $(\mathfrak{g},K)$-modules of $U(p,q)$}

By Langlands classification, all irreducible $(\mathfrak{g},K)$-modules of real reductive groups can be realized as a subquotient of an induced module from a discrete series twisted with a (non-unitary) character. Using the cohomological induction, one can construct the discrete series by taking cohomological induction from a character. Hence, one can realize all irreducible $(\mathfrak{g},K)$-modules as a subquotient of the parabolic induction of some module coming from cohomological induction. 

By \cite[Theorem 6.6.15]{V81}, one can exchange the order of 
the real parabolic induction and the cohomological induction under some conditions. Hence, one can first do the parabolic induction and then the cohomological induction.
For the purpose of classifying unitary dual, one would like to use the latter way to construct the modules, and also call it the Langlands classification. 

\subsection{Combinatorial $\theta$-stable data}

Let $G=U(p,q)$ and $\mathfrak g$ be the complexification of the Lie algebra of $G$. The maximal compact subgroup of $G$ is $K\cong U(p)\times U(q)$. Let $\mathfrak g=\mathfrak k\oplus\mathfrak p$ be the corresponding Cartan decomposition on the Lie algebra level. Denote $T$ to be a maximal torus in $K$ and $\mathfrak t$ to be its complexified Lie algebra. There are naturally two root systems, $\Delta(\mathfrak g,\mathfrak t)$ for $\mathfrak g$ and $\Delta(\mathfrak k,\mathfrak t)$ for $\mathfrak k$. Choose positive roots so that 
 $\Delta^+(\mathfrak k,\mathfrak t)\subset\Delta^+(\mathfrak g,\mathfrak t)$. Denote $\rho$ (resp. $\rho_c$) to be the half sum of positive roots of $\mathfrak g$ (resp. $\mathfrak k$). For a $\theta$-stable parabolic subalgebra $\mathfrak{q}=\mathfrak{l}\oplus\mathfrak{u}$, denote $\rho(\mathfrak u)$ to be the half sum of roots in $\mathfrak u$, and same to $\rho(\mathfrak u\cap \mathfrak p)$.

The Langlands classification has a combinatorial realization for the case of $U(p,q)$, see \cite{Wong24} for more details. Moreover, the unitary dual of $U(n,2)$ is completely understood using the combinatorial data, see \cite{WongZhang24}. Briefly speaking, any irreducible $(\mathfrak g,K)$-modules $X$ can be constructed from a pair $(\lambda_a(\delta),\nu)$. Here $\delta$ is a lowest $K$-type of $X$, and $\lambda_a(\delta)$ is expressed by a $\lambda_a$-datum which is obtained from $\delta$ through a specific algorithm, see \cite{Wong24}. In this subsection, we will explain the format of the combinatorial method in \cite{Wong24} so that we can quote the results in \cite{WongZhang24}.

We start with $\gamma$-blocks which are constituents of $\lambda_a$-datum. The diagram defined below will be used to represent the lowest $K$-type.

\begin{definition}\label{def.gamma.block}
    Let $\gamma$ be in $\frac{1}{2}\mathbb Z$, and let $G=U(p,q)$ with $p+q\equiv \varepsilon$ $(\emph{mod}$ $ 2)$ $(\varepsilon=0$ or $1)$. A $\gamma$\textbf{-block} of size $(r,s)$, where $0\leqslant r\leqslant p$, $0\leqslant s\leqslant q$ are non-negative integers satisfying $|r-s|\leqslant 1$, is one of the following shapes with $\gamma$ attached to them.
\end{definition}

\begin{itemize}
\item {\bf Rectangle} of size $(r,r)$:
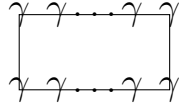
\begin{center}
\begin{tikzpicture}
\draw
    (0,0) node {\Large $\gamma$}
 -- (0.5,0) node {\Large $\gamma$}		
 -- (1,0) node {\Large $\ldots$}
 -- (1.5,0) node {\Large $\gamma$}
 -- (2,0) node {\Large $\gamma$}
 -- (2,-1) node {\Large $\gamma$}
 -- (1.5,-1) node {\Large $\gamma$}
-- (1,-1) node {\Large $\ldots$}
-- (0.5,-1) node {\Large $\gamma$}
-- (0,-1) node {\Large $\gamma$}
 -- cycle;
\end{tikzpicture} 
\end{center}
where $\gamma + \frac{\epsilon}{2} \in \mathbb{Z}$.

\item {\bf Parallelogram} of size $(r,r)$:
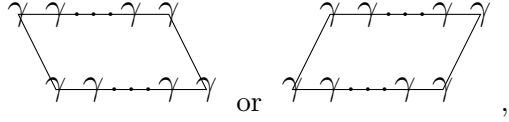
\begin{center}
\begin{tikzpicture}
\draw
    (0,0) node {\Large $\gamma$}
 -- (0.5,0) node {\Large $\gamma$}		
 -- (1,0) node {\Large $\ldots$}
 -- (1.5,0) node {\Large $\gamma$}
 -- (2,0) node {\Large $\gamma$}
 -- (2.5,-1) node {\Large $\gamma$}
 -- (2,-1) node {\Large $\gamma$}
-- (1.5,-1) node {\Large $\ldots$}
-- (1,-1) node {\Large $\gamma$}
-- (0.5,-1) node {\Large $\gamma$}
 -- cycle;
\end{tikzpicture} or 
\begin{tikzpicture}
\draw
    (0,0) node {\Large $\gamma$}
 -- (0.5,0) node {\Large $\gamma$}		
 -- (1,0) node {\Large $\ldots$}
 -- (1.5,0) node {\Large $\gamma$}
 -- (2,0) node {\Large $\gamma$}
 -- (1.5,-1) node {\Large $\gamma$}
 -- (1,-1) node {\Large $\gamma$}
-- (0.5,-1) node {\Large $\ldots$}
-- (0,-1) node {\Large $\gamma$}
-- (-0.5,-1) node {\Large $\gamma$}
 -- cycle;
\end{tikzpicture},
\end{center}
where $\gamma + \frac{\epsilon + 1}{2} \in \mathbb{Z}$.

\item {\bf Trapezoid} of size $(r,r-1)$ or $(r,r+1)$:
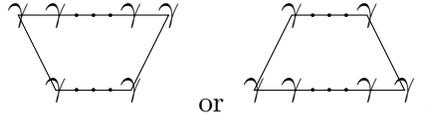
\begin{center}
\begin{tikzpicture}
\draw
    (0,0) node {\Large $\gamma$}
 -- (0.5,0) node {\Large $\gamma$}		
 -- (1,0) node {\Large $\ldots$}
 -- (1.5,0) node {\Large $\gamma$}
 -- (2,0) node {\Large $\gamma$}
 -- (1.5,-1) node {\Large $\gamma$}
-- (1,-1) node {\Large $\ldots$}
-- (0.5,-1) node {\Large $\gamma$}
 -- cycle;
\end{tikzpicture} or 
\begin{tikzpicture}
\draw
 (0.5,0) node {\Large $\gamma$}		
 -- (1,0) node {\Large $\ldots$}
 -- (1.5,0) node {\Large $\gamma$}
 -- (2,-1) node {\Large $\gamma$}
 -- (1.5,-1) node {\Large $\gamma$}
-- (1,-1) node {\Large $\ldots$}
-- (0.5,-1) node {\Large $\gamma$}
-- (0,-1) node {\Large $\gamma$}
 -- cycle;
\end{tikzpicture},
\end{center}
where $\gamma + \frac{\epsilon + 1}{2} \in \mathbb{Z}$.

\end{itemize}

\begin{definition}\label{def.lambda.datum}
A $\lambda_a$\textbf{-datum} for $U(p,q)$ is a collection of $\gamma_i$-blocks of sizes $(r_i,s_i)$ such that $\sum r_i=p$, $\sum s_i=q$ and all $\gamma_i$ are distinct.
\end{definition}

The following theorem shows the usage of the $\lambda_a$-datum, which is the well-known
bijection between $\widehat{K}$ and all tempered representations with real infinitesimal characters, see also \cite[Section 3]{Wong24}. 
\begin{theorem}\label{thm-datum}
    Let $G = U(p,q)$. There is a $1:1$-correspondence between the following sets:
\begin{itemize}
    \item[(a)] All $\lambda_a$-data of $G$.
    \item[(b)] All $K$-types $\delta \in \widehat{K}$.
    \item[(c)] All tempered representations with real infinitesimal characters.
\end{itemize}
\end{theorem}

If a $\lambda_a$-datum is obtained from a $K$-type $\delta$, we write the datum as $\lambda_a(\delta)$. A $\lambda_a$-datum determines a $\theta$-stable quasisplit parabolic subalgebra $\mathfrak q_0=\mathfrak l_0\oplus\mathfrak u_0$. Suppose a $\lambda_a$-datum consists of $\gamma_i$-blocks ($1\leqslant i\leqslant m$) of size $(r_i,s_i)$ with $\gamma_i>\gamma_{i+1}$. Then the Levi subalgebra $\mathfrak l_0$ is $\mathfrak u(r_1,s_1)\oplus\cdots\oplus\mathfrak u(r_m,s_m)$. The core notion we will frequently use in this paper is the combinatorial $\theta$-stable datum defined as follows.

\begin{definition}
    Let $G=U(p,q)$, and $X$ be an irreducible Hermitian $(\mathfrak{g}, K)$-module with a lowest $K$-type $\delta$ and real infinitesimal character $\Lambda(X)=(\lambda_a(\delta);\nu)$, where $\lambda_a(\delta)\in \mathfrak{t}^{\prime\ast}$ and $\nu\in\mathfrak{a}^{\prime\ast}$, and $\mathfrak{t}^{\prime}\oplus \mathfrak{a}^{\prime}$ is the complexification of a maximal non-compact Cartan subalgebra of $\mathfrak{g}_0$. A \textbf{combinatorial $\theta$-stable datum} attached to $X$ is given by the following two components:\\
\indent (1) A $\lambda_a$-datum determined by $\delta$; and \\
\indent (2) For each $\gamma_i$-block of size $(r_i,s_i)$ in (1), an element $\nu_i\in\mathbb R^{\min\{r_i,s_i\}}$ of the form $\nu_i=(\nu_{i,1}\geqslant \nu_{i,2} \geqslant \cdots \geqslant -\nu_{i,2}\geqslant -\nu_{i,1})$ $($the $\nu_i$'s are called the $\nu$-\textbf{coordinates corresponding to the $\gamma_i$ -block}$)$, so that $\nu=(\nu_1,\nu_2,\cdots)$ up to conjugacy.
\end{definition}

Each combinatorial datum attached to $X$ determines an induced $(\mathfrak g,K)$-module, so
that $X$ is the irreducible subquotient of the induced module containing $\delta$ as a lowest $K$-type. Here we give an example of a combinatorial $\theta$-stable datum, and we will mainly explain how to rewrite the infinitesimal character from a datum. 

\begin{example} \label{eg-lambdaa2}

Let $G = U(6,3)$, and consider the combinatorial $\theta$-stable datum with the $\nu$-coordinates as follows. For the $1$-block, let $\nu_1=(\frac{3}{4},\frac{1}{2},-\frac{1}{2},-\frac{3}{4})$, and for $0$-block, let $\nu_2=(\frac{1}{3},-\frac{1}{3})$. We will describe the $\nu$-coordinates by a curved arrow with $\nu$-coordinates starting from the corresponding block:
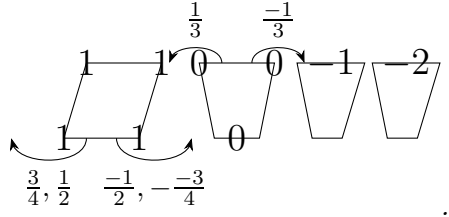
\begin{center}
\begin{tikzpicture}
\draw
    (0,0) node {\Large $1$}
 -- (1,0) node {\Large $1$}
-- (0.7,-1) node {\Large $1$}
-- (-0.3,-1) node {\Large $1$}
 -- cycle;

\draw
    (1.5,0) node {\Large $0$}
 -- (2.5,0) node {\Large $0$}
-- (2.3,-1) 
-- (2,-1) node {\Large $0$}
-- (1.7,-1) 
 -- cycle;

\draw
    (2.8,0) 
 -- (3.25,0) node {\Large $-1$}		
-- (3.7,0)
 -- (3.4,-1) 
-- (3,-1) 
 -- cycle;

\draw
    (3.8,0) 
 -- (4.25,0) node {\Large $-2$}		
-- (4.7,0)
 -- (4.4,-1) 
-- (4,-1) 
 -- cycle;

\draw[arrows = {-Stealth[]}]          (0,-1)   to [out=-90,in=-90]node[below]{$\frac{3}{4},\frac{1}{2}$} (-1,-1);
\draw[arrows = {-Stealth[]}]          (0.4,-1)   to [out=-90,in=-90]node[below]{$\frac{-1}{2},-\frac{-3}{4}$} (1.4,-1);

\draw[arrows = {-Stealth[]}]          (1.8,0)   to [out=90,in=90]node[above]{$\frac{1}{3}$} (1.1,0);
\draw[arrows = {-Stealth[]}]          (2.2,0)   to [out=90,in=90]node[above]{$\frac{-1}{3}$} (2.9,0);
\end{tikzpicture}.
\end{center}
Then it defines a $\theta$-stable quasisplit parabolic subalgebra 
\begin{equation} \label{eq-thetaeg}
\mathfrak{q}_0 = \mathfrak{l}_0 \oplus \mathfrak{u}_0, \quad \quad \mathfrak{l}_0 = \mathfrak{u}(2,2) \oplus \mathfrak{u}(2,1) \oplus \mathfrak{u}(1,0) \oplus \mathfrak{u}(1,0).
\end{equation}
The corresponding module has infinitesimal character
$$
\Lambda=\lambda_a+\nu=\left(\left.1+\frac{3}{4},1+\frac{1}{2},0+\frac{1}{3},0,-1,-2\right|1-\frac{1}{2},1-\frac{3}{4},0-\frac{1}{3}\right),
$$
and the arrows show how the addition works.
\end{example}

\subsection{Partitions of combinatorial $\theta$-stable data}

Given a combinatorial $\theta$-stable datum $\mathcal D=\coprod (\gamma_i$-block), a partition of $\mathcal D$ is a regroup of the blocks in which only consecutive blocks are grouped together in the partition. Using the notion of partitions of a datum, \cite{Wong24} described the condition for a representation to be fully supported, see Lemma \ref{lemma.fully.supported.condition} below. And in \cite{WongZhang24} the authors reduce the problem of the unitary dual of $U(n,2)$ to fundamental cases.

Let $X$ be an irreducible, Hermitian $(\mathfrak g,K)$-module with real infinitesimal character and a lowest $K$-type $\delta$. All $\theta$-stable parabolic subalgebras $\mathfrak q^\prime$ containing the quasisplit parabolic subalgebra $\mathfrak q$ defined by $\lambda_a(\delta)$ are in 1-1 correspondence with the set of partitions of the $\lambda_a$ in the combinatorial $\theta$-stable datum of $X$. As a consequence, all such $\theta$-stable parabolic subalgebras $\mathfrak q^\prime \supset \mathfrak q$ correspond to a partition $\mathcal D=\coprod_i^k D_i$ of the $\theta$-stable datum of $X$, and the infinitesimal character $\Lambda=(\lambda_a(\delta),\nu)$ of $X$ is also partitioned into $(\Lambda_1,\cdots,\Lambda_k)$ up to a permutation of coordinates. Define the $i^{th}$-\emph{segment} of the partition by the line segment $[e_i,b_i]$ where $e_i$ (resp. $b_i$) is the largest (resp. smallest) number in $\Lambda_i$.

\begin{example}
    In example \ref{eg-lambdaa2}, the following is partition $\mathcal D=\mathcal D_1\coprod \mathcal D_2$ of the datum:
$$\mathcal D_1:=\{1\text{-}\mathrm{block},0\text{-}\mathrm{block} \}\qquad
\mathcal D_2:=\{(-1)\text{-}\mathrm{block},(-2)\text{-}\mathrm{block} \}.$$
Notice that $\nu$ is $(\frac{3}{4},\frac{1}{2})$ for the 1-block and $\frac{1}{3}$ for the 0-block. The infinitesimal character is partitioned into
$$\Lambda_1=\left(
1+\frac{3}{4},1+\frac{1}{2},1-\frac{1}{2},1-\frac{3}{4},0+\frac{1}{3},0,0-\frac{1}{3}
\right)=
\left(
\frac{7}{4},\frac{3}{2},\frac{1}{2},\frac{1}{3},\frac{1}{4},0,-\frac{1}{3}
\right);$$
$$\Lambda_2=(-1,-2).$$
Hence the segments are $[e_1,b_1]=[\frac{7}{4},-\frac{1}{3}]$ and $[e_2,b_2]=[-1,-2]$.
\end{example}

\begin{lemma}\label{lemma.fully.supported.condition}\cite[Section 7]{Wong24}
Retain the setting in the above paragraphs. Then $X$ is cohomologically induced from a $\theta$-stable parabolic subalgebra $\mathfrak q^\prime$ in the weakly good range if and only if the segments of the partition of the combinatorial $\theta$-stable datum of $X$ corresponding to $\mathfrak q^\prime$ satisfy the following inequalities:
$$e_1\geqslant b_1\geqslant e_2\geqslant b_2\geqslant \cdots\geqslant e_k\geqslant b_k.$$
\end{lemma}

If an irreducible module is not fully supported, then it can be obtained by cohomological induction from a proper $\theta$-stable parabolic subalgebra in the weakly good range. Equivalently, if an irreducible module cannot be obtained by cohomological induction from a proper $\theta$-stable parabolic subalgebra in the weakly good range, then it must be fully supported. The definition of being fully supported is related to the KGB-elements which are frequently used in the software \emph{atlas}. In the language of \emph{atlas}, an irreducible representation can be realized by a triple $(x,\lambda,\nu)$, where $x$ is a KGB-element. If the support of the KGB-element contains all simple roots of $G$, then the representation is called fully supported.

\begin{definition}
    Let $G$ be any real connected reductive group. An irreducible, Hermitian $(\mathfrak g,K)$-module with real infinitesimal character is called \textbf{fundamental} if its $\lambda_a$-value satisfies $\langle \lambda_a,\alpha\rangle\leqslant1$ for all simple roots $\alpha\in \Delta(\mathfrak g,\mathfrak h)$.
\end{definition}

When $G=U(p,q)$, there is an equivalent definition in the sense of combinatorial $\theta$-stable datum.

\begin{definition}
    Let $G=U(p,q)$. A $\lambda_a$-datum or a combinatorial $\theta$-stable datum is called \textbf{fundamental} if the contents of all its neighboring $\lambda_a$-blocks have differences $\leqslant 1$.
\end{definition}

It is clear that there is a unique way to write the $\lambda_a$-data as a union of minimal number of fundamental data. According to such partition, one has the following theorem.

\begin{theorem}\cite[Section 4]{WongZhang24}\label{thm.weakly.good.ind.unitary}
    Let $G=U(n,2)$, and $X$ be a Hermitian, irreducible module with real infinitesimal character. Suppose its corresponding combinatorial $\theta$-stable datum $\mathcal D=\coprod_{i=1}^k \mathcal F_i$ is partitioned into fundamental data. Then $X$ is unitary if and only if each fundamental datum $\mathcal F_i$ corresponds to a unitary module.
\end{theorem}

The decomposition in fact realizes how $X$ is constructed by the cohomological induction. Firstly, we locate the $\theta$-stable parabolic subalgebra $\mathfrak q$. Suppose the size of the $i^{th}$-segment is $(r_i,s_i)$, then the $i^{th}$ factor of the Levi subalgebra $\mathfrak l_0$ is $\mathfrak u(r_i,s_i)$. Secondly, we locate the representation $\pi$ on $L\cap K$. The restriction of $\pi$ on the $i^{th}$ Levi factor is the one corresponding to the underlying combinatorial $\theta$-stable datum $\mathcal F_i$.

\subsection{Fundamental cases of $U(n,1)$ and $U(n,2)$}

The notion of fundamental representations is introduced in \cite{Wong24} in order to prove Vogan's fundamental parallelepiped conjecture. The conjecture is proved firstly in the fundamental modules and then in the general cases. In the case of $U(n,2)$ and $U(n,1)$, fully supported representations are fundamental.

In \cite{WongZhang24}, there is a full list of  unitary fundamental $\lambda_a$-data for $U(n,1)$ and $U(n,2)$, together with the equivalent conditions on $\nu$. We will only list below the unitary fundamental cases which are  fully supported.

\begin{example}\label{eg.U(n1).fundamental.cases}
The following datum includes most of the fundamental cases for $U(n,1)$.
    
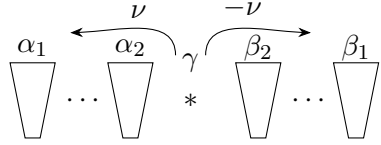
\begin{center}
    \begin{tikzpicture}
\foreach \y in {1, 2.3, 4, 5.3}
    \draw (\y,1)
    -- (\y+0.2,0)
    -- (\y+0.4,0)
    -- (\y+0.6,1)
    -- cycle;

\foreach \y in {2,5}
    \draw (\y,0.5) node {$\cdots$};

\draw (3.4,0.5) node {$\ast$};
\draw (3.4,1) node {$\gamma$};
\draw (1.3,1.2) node {$\alpha_1$};  \draw (2.6, 1.2) node {$\alpha_2$};
\draw (5.6,1.2) node {$\beta_1$};   \draw (4.3, 1.2) node {$\beta_2$};


\draw[arrows = {-Stealth[]}]          (3.2,1.1)   to [out=90,in=0]node[above]{$\nu$} (1.8,1.4);
\draw[arrows = {-Stealth[]}]          (3.6,1.1)   to [out=90,in=180]node[above]{$-\nu$} (5,1.4);

\end{tikzpicture}
\end{center}
Here the $\gamma$-block at $\ast$ is either a trapezoid of size $(2,1)$ or a rectangle of size $(1,1)$. We have
\begin{equation}
    \lambda_a=\big(~\alpha_1,\alpha_1-1,\cdots,\gamma,\cdots,\beta_1+1,\beta_1~\big|~\gamma~\big).
\end{equation}
The values on any two neighboring trapezoids of size $(1,0)$ have difference $=1$. When the $\gamma$-block is a trapezoid of size $(2,1)$, we have $|\gamma-\alpha_2|=|\gamma-\beta_2|=1$; and when the $\gamma$-block is a rectangle of size $(1,1)$, we have $|\gamma-\alpha_2|=|\gamma-\beta_2|=\frac{1}{2}$.

The corresponding $(\mathfrak{g},K)$-module is unitary when $0\leqslant \nu \leqslant \min\{ |\gamma-\alpha_1|+1, |\gamma-\beta_1|+1 \}$; And if it is further fully supported, it has to be the trivial module who has $\nu=|\gamma-\alpha_1|+1=|\gamma-\beta_1|+1$.

\end{example}

\begin{example}\label{eg.u(n2).fundamental.cases}
The following data include most of the fundamental cases for $U(n,2)$. Type (a) has $k+l+m$ $(1,0)$-blocks and two $(r,1)$-blocks where $r=1,2$, where at most one of them is a $(1,1)$-parallelogram. By tensoring a power of the 1-dimensional determinant representation, one can assume the first and last block have sum zero. 
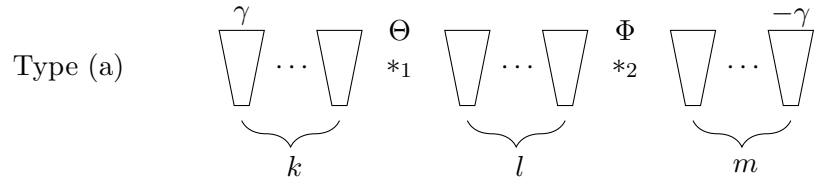
\begin{center}
    \begin{tikzpicture}
\draw (-1,0.5) node {\emph{Type (a)}};

\foreach \y in {1, 2.3, 4, 5.3, 7, 8.3}
    \draw (\y,1)
    -- (\y+0.2,0)
    -- (\y+0.4,0)
    -- (\y+0.6,1)
    -- cycle;

\foreach \y in {2, 5, 8}
    \draw (\y,0.5) node {$\cdots$};

\draw (3.4,0.5) node {$\ast_1$};  \draw (6.4,0.5) node {$\ast_2$};
\draw (3.4,1) node {$\Theta$};  \draw (6.4,1) node {$\Phi$};
\draw (1.3,1.2) node {$\gamma$};
\draw (8.6,1.2) node {$-\gamma$};

\foreach \x in {1.3, 4.3, 7.3}
    \draw decorate [decoration={brace, amplitude=10pt, mirror}]{(\x,-0.2) -- (\x+1.3,-0.2)};

\draw (2,-0.8) node {$k$}; \draw (5,-0.8) node {$l$}; \draw (8,-0.8) node {$m$};

\end{tikzpicture}
\end{center}

Type (b) has $k+m$ $(1,0)$-blocks and one $(r,2)$-block $(r=2,3)$.

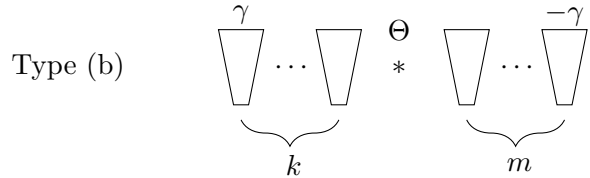
\begin{center}
    \begin{tikzpicture}
\draw (-1,0.5) node {\emph{Type (b)}};

\foreach \y in {1, 2.3, 4, 5.3}
    \draw (\y,1)
    -- (\y+0.2,0)
    -- (\y+0.4,0)
    -- (\y+0.6,1)
    -- cycle;

\foreach \y in {2, 5}
    \draw (\y,0.5) node {$\cdots$};

\draw (3.4,0.5) node {$\ast$};
\draw (3.4,1) node {$\Theta$};
\draw (1.3,1.2) node {$\gamma$};
\draw (5.6,1.2) node {$-\gamma$};

\foreach \x in {1.3, 4.3}
    \draw decorate [decoration={brace, amplitude=10pt, mirror}]{(\x,-0.2) -- (\x+1.3,-0.2)};

\draw (2,-0.8) node {$k$}; \draw (5,-0.8) node {$m$};

\end{tikzpicture}
\end{center}
For type (a), we use $\nu_\Theta$ and $\nu_\Phi$ to denote the arrows which come out from $\ast_1$ and $\ast_2$. For type (b), we may consider it as a special case of type (a), that is when $\Theta=\Phi$. There are two pairs of arrows coming out from $\ast$. We still use the notation $\nu_\Theta$ and $\nu_\Phi$ in this case.

Here are the conditions on $\nu$ when the combinatorial $\theta$-stable datum corresponds to an irreducible unitary representation. For a datum of type (a), let one of the $\ast$ positions be a $(1,1)$-parallelogram and the other $\ast$ position be a $(1,1)$-rectangle or a $(2,1)$-trapezoid. And suppose the parallelogram leans to the other $\ast$ position.

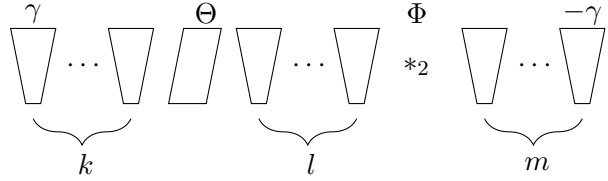
\begin{center}
    \begin{tikzpicture}

\foreach \y in {1, 2.3, 4, 5.3, 7, 8.3}
    \draw (\y,1)
    -- (\y+0.2,0)
    -- (\y+0.4,0)
    -- (\y+0.6,1)
    -- cycle;

\foreach \y in {2, 5, 8}
    \draw (\y,0.5) node {$\cdots$};

\draw (6.4,0.5) node {$\ast_2$};
\draw (3.6,1.2) node {$\Theta$};  \draw (6.4,1.2) node {$\Phi$};
\draw (1.3,1.2) node {$\gamma$};
\draw (8.6,1.2) node {$-\gamma$};

\foreach \y in {3.4}
    \draw (\y-0.1,1) -- (\y-0.3, 0) -- (\y+0.2, 0) -- (\y+0.4, 1) -- cycle;

\foreach \x in {1.3, 4.3, 7.3}
    \draw decorate [decoration={brace, amplitude=10pt, mirror}]{(\x,-0.2) -- (\x+1.3,-0.2)};

\draw (2,-0.8) node {$k$}; \draw (5,-0.8) node {$l$}; \draw (8,-0.8) node {$m$};

\end{tikzpicture}
\end{center} 

Without loss of generality, we assume $\ast_1$ is a $(1,1)$-parallelogram leaning to $\ast_2$, while $\ast_2$ is a $(1,1)$-rectangle or a $(2,1)$-trapezoid. When $\ast_2$ has size $(r,1)$, this datum corresponds to an irreducible unitarizable $(\mathfrak{g},K)$-module if and only if $\nu_\Theta=0$ and $0\leqslant \nu_\Phi \leqslant \min \{ l+1,m \}+\frac{r}{2}$. See \cite[Subsection 4.3]{WongZhang24} for details.

Suppose neither $\ast_1$ nor $\ast_2$ is a parallelogram, then the combinatorial $\theta$-stable datum corresponds to an irreducible unitary $(\mathfrak{g},K)$-module if and only if one of the following two conditions is satisfied:

(i) If $(\nu_\Theta,\nu_\Phi)$ is in the fundamental rectangle, see \cite[Subsection 4.2]{WongZhang24},
\begin{equation}\label{fundamental.rectangle}
    \left\{
    (\nu_{\Theta}, \nu_{\Phi}) \left|~0\leqslant \nu_{\Theta}\leqslant \frac{n-1}{2}-|\Theta|,~0\leqslant \nu_{\Phi}\leqslant \frac{n-1}{2}-|\Phi|\right.
    \right\},
\end{equation} 
then it is either in the triangle areas bounded by the inequalities
\begin{itemize} 
    \item $\{\nu_{\Theta} + \nu_{\Phi}\leqslant \Theta-\Phi+1\}$;  
    \item $\{\nu_{\Theta} + \nu_{\Phi} \leqslant \Theta-\Phi+k+1, \nu_{\Theta} - \nu_{\Phi}\geqslant \Theta-\Phi+k\}, k\in \mathbb{N}_+ $
    \item $\{\nu_{\Theta} + \nu_{\Phi} \leqslant \Theta-\Phi+k+1, \nu_{\Phi} - \nu_{\Theta}\geqslant \Theta-\Phi+k\}, k\in \mathbb{N}_+ $
\end{itemize}
or it is on the lines:
\begin{itemize}
    \item $\{\nu_{\Theta} - \nu_{\Phi}=\Theta-\Phi+k\}, k\in \mathbb{N}_+$. 
    \item $\{\nu_{\Phi} - \nu_{\Theta}=\Theta-\Phi+k\}, k\in \mathbb{N}_+$.
\end{itemize}
(ii) If $(\nu_\Theta,\nu_\Phi)$ is not in the fundamental rectangle
, then it is one of the following points (with $\Theta=\Phi$ or $\Theta\cdot\Phi>0$)
\begin{equation}
(\nu_{\Theta}, \nu_{\Phi}) =
 \begin{cases}
 (\frac{n-1}{2}-|\Theta|, \frac{n+1}{2}-|\Phi|);~
 (\frac{n+1}{2}-|\Theta|, \frac{n-1}{2}-|\Phi|) & \mathrm{if}\ \Theta = \Phi,\\
 (\frac{n-1}{2}-|\Theta|, \frac{n+1}{2}-|\Phi|) & \mathrm{if}\  \Theta > \Phi>0\\ (\frac{n+1}{2}-|\Theta|, \frac{n-1}{2}-|\Phi|) & \mathrm{if}\ 0 > \Theta >\Phi\\
\end{cases}.
\end{equation}
\end{example}

\begin{example}\label{eg.u22}
    Let $G=U(2,2)$. Consider the fundamental data
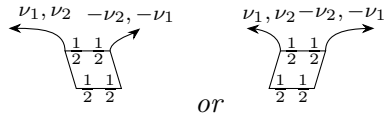
\begin{center}
\begin{tikzpicture}
\draw
    (-0.15,0) 
 -- (-0.3,0) node {\scriptsize $\frac{1}{2}$}
 -- (-0.45, 0) 
-- (-0.6,0) node {\scriptsize $\frac{1}{2}$}
-- (-0.75,0)
-- (-0.6,-0.5) 
-- (-0.45,-0.5) node {\scriptsize $\frac{1}{2}$}
-- (-0.3,-0.5)
-- (-0.15,-0.5) node {\scriptsize $\frac{1}{2}$}
-- (0,-0.5)
 -- cycle;

 \draw[arrows = {-Stealth[]}]          (-0.15,0)   to [out=90,in=00]node[above]{\scriptsize $-\nu_2,-\nu_1$} (0.25,0.3);
\draw[arrows = {-Stealth[]}]          (-0.75,0)   to [out=90,in=00]node[above]{\scriptsize $\nu_1,\nu_2$} (-1.5,0.3);
\end{tikzpicture}
or  \begin{tikzpicture}
\draw
    (0.15,0) 
 -- (0.3,0) node {\scriptsize $\frac{1}{2}$}
 -- (0.45, 0) 
-- (0.6,0) node {\scriptsize $\frac{1}{2}$}
-- (0.75,0)
-- (0.6,-0.5) 
-- (0.45,-0.5) node {\scriptsize $\frac{1}{2}$}
-- (0.3,-0.5)
-- (0.15,-0.5) node {\scriptsize $\frac{1}{2}$}
-- (0,-0.5)
 -- cycle;

 \draw[arrows = {-Stealth[]}]          (0.15,0)   to [out=90,in=00]node[above]{\scriptsize $\nu_1,\nu_2$} (-0.3,0.3);
\draw[arrows = {-Stealth[]}]          (0.75,0)   to [out=90,in=180]node[above]{\scriptsize $-\nu_2,-\nu_1$} (1.4,0.3);
\end{tikzpicture}
\end{center}
with $\nu_1\geqslant \nu_2\geqslant 0$. Then both data correspond to unitary modules if and only if $\nu_1+\nu_2\leqslant 1$. In particular, when $\nu_2>0$, they correspond to the same module. When $\nu_2=0, \nu_1=1$, they are fully supported $A_\mathfrak{q}(\lambda)$ modules. For the left one, the parabolic subalgebra $\mathfrak q$ has Levi subalgebra $\mathfrak u(2,1)\oplus\mathfrak u(0,1)$ and $\lambda=(0,0,0|2)$; for the right one, the parabolic subalgebra $\mathfrak q$ has Levi subalgebra $\mathfrak u(1,2)\oplus\mathfrak u(1,0)$ and also $\lambda=(0,0,0|2)$.

$U(2,2)$ has two more fully supported $A_\mathfrak q(\lambda)$ modules. One is the trivial representation. The other corresponds to combinatorial $\theta$-stable datum
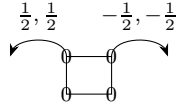
\begin{center}
    \begin{tikzpicture}
\draw
    (-0.15,0) 
 -- (-0.15,0) node {\scriptsize $0$}
 -- (-0.45,0) 
-- (-0.45,0) 
-- (-0.75,0) node {\scriptsize $0$}
-- (-0.75,-0.5) 
-- (-0.75,-0.5) node {\scriptsize $0$}
-- (-0.15,-0.5)
-- (-0.15,-0.5) node {\scriptsize $0$}
-- (-0.15,-0.5)
 -- cycle;

 \draw[arrows = {-Stealth[]}]          (-0.15,0)   to [out=90,in=90]node[above]{\scriptsize $-\frac{1}{2},-\frac{1}{2}$} (0.6,0);
\draw[arrows = {-Stealth[]}]          (-0.75,0)   to [out=90,in=90]node[above]{\scriptsize $\frac{1}{2},\frac{1}{2}$} (-1.5,0);
\end{tikzpicture}
\end{center}
and $\mathfrak{q}$ has Levi subalgebra $\mathfrak{u}(1,1)\oplus\mathfrak{u}(1,1)$, and $\lambda=(-1,-1|1,1)$.

In total, after tensoring a power of the determinant representation, there are 4 fully supported unitary representations of $U(2,2)$.
\end{example}

\section{The Dirac cohomology and the finiteness theorem}

Let $U(\mathfrak g)$ be the universal enveloping algebra of $\mathfrak g$ and $C(\mathfrak p)$ be the Clifford algebra of $\mathfrak p$. Fix a standard nondegenerate bilinear form $B(\cdot, \cdot)$ on $\mathfrak g$, and take an orthonormal basis $Z_1,Z_2,\cdots,Z_{\dim \mathfrak p_0}$ of $\mathfrak p_0$. The Dirac operator is defined to be
\begin{equation*}
    D:=\sum Z_i\otimes Z_i\in U(\mathfrak g)\otimes C(\mathfrak p).
\end{equation*}
It is easy to check that $D$ does not depend on the choice of the orthonormal basis, and it is $K$-invariant for the diagonal action of $K$ given by adjoint actions on both factors. For any given $(\mathfrak g,K)$-module, the Dirac operator acts on the tensor product $X\otimes \mathrm{Spin}_G$, where $\mathrm{Spin}_G$ is the spin module of $C(\mathfrak{p})$. The \emph{Dirac cohomology} is defined as
$$H_D(X):=\ker D/(\mathrm{im} D\cap \ker D).$$
In the case when $X$ is unitary, $(\mathrm{im} D\cap \ker D)=0$ and hence $H_D(X)=\ker D^2=\ker D$. The Dirac cohomology, if not zero, gives the infinitesimal character of the original (irreducible) module.

\begin{theorem}\cite[Theorem 2.3]{HP2002}\label{thm.HP.condition}
    Let $X$ be an irreducible $(\mathfrak g,K)$-module. Assume that the Dirac cohomology of $X$ contains a $\widetilde K$-type $E_\mu$ of highest weight $\mu\in\mathfrak t^\ast$. Then the infinitesimal character of $X$ is $\mu+\rho_c$.
\end{theorem}

An important way to construct an irreducible unitary $(\mathfrak g,K)$-module is using cohomological parabolic induction. There is a formula to compute the Dirac cohomology for cohomologically induced modules proved by Dong and Huang.

\begin{theorem}\cite[Theorem 5.7]{DongHuang15}\label{thm.HD.induction}
    Let $G$ be a real reductive Lie group in the Harish-Chandra class. Let $\mathfrak q = \mathfrak l\oplus\mathfrak u$ be a $\theta$-stable parabolic subalgebra with the Levi subgroup $L = N_G(\mathfrak q)$. Let $\mathcal L_S(\cdot)$ denote the cohomological induction functor associated with $\mathfrak q$. Suppose the irreducible unitary $(\mathfrak l,L\cap K)$-module $Z$ is weakly good. Then there is a $\widetilde K$-module isomorphism
    \begin{equation*}
        H_D(\mathcal L_S(Z))\cong \mathcal L^{\widetilde K}_S(H_D(Z)\otimes \mathbb{C}_{-\rho(\mathfrak u\cap\mathfrak p)})
    \end{equation*}
\end{theorem}

We will not use the explicit formula but its direct corollary.

\begin{corollary}\label{cor.HD.induction}
    Under the same condition of theorem \ref{thm.HD.induction}, let $X$ be cohomologically induced from $Z$ in the weakly good range. Then $H_D(X)\neq 0$ if and only if $H_D(Z)\neq 0$.
\end{corollary}

If one requires the group $G$ has finite center, then Dong gives a clearer description of the Dirac series of $G$, which is the finiteness theorem.

\begin{theorem}\cite[Theorem A]{Dong20}\label{thm.HD.finiteness}
    Let $G$ be a real group of finite center. All but finitely many members of the Dirac series is cohomologically induced from a member of the Dirac series of $L$ that is in the good range. Here $L$ is a proper $\theta$-stable Levi subgroup of $G$.
\end{theorem}

The author calls the Dirac series obtained cohomological induction from good range the \emph{string representations}, and the ones that cannot be obtained in this way are called the \emph{scattered representations}. Later in his another paper \cite{DDY20}, the author realized that the scattered representations which are not fully supported should be viewed as the starting points of strings of singling them out. Then what remains are the fully supported scattered part, which are called FS-scattered representations.

Because the FS-scattered representations cannot be obtained by cohomological induction from the weakly good range, they are all fully supported. So, people may also write them as FS-scattered representations. Moreover, the FS-scattered representations are the building blocks of the string representations.

\section{The Dirac series of $U(n,2)$}

\subsection{Conditions on the infinitesimal characters}

Let $X$ be an FS-scattered representation of $U(n,2)$ in the Dirac series. Then it is a fully supported representation with nonzero Dirac cohomology. In this subsection, we show some necessary conditions on the infinitesimal character $\Lambda(X)$ of $X$.

Theorem \ref{thm.HP.condition} implies a condition of the infinitesimal character of the irreducible modules having non-zero Dirac cohomology, which is called as H.P.-condition. The following Lemma can be found in \cite[Lemma 5.2]{DongWong21}, which is the H.P-condition in the $U(p,q)$ case. 

\begin{lemma}\label{lemma.H.P.condition}(H.P.-condition)
    Let $X$ be a representation of $U(n,2)$ with nonzero Dirac cohomology $H_D(X)$. Then the infinitesimal character $\Lambda(X)$ of $X$ is in $\rho+\mathbb Z^{n+2}$ (integral). Moreover, if we write $\Lambda(X)$ as
    \begin{equation}\label{Lambda.usual.coor}
        \Lambda(X)=(x_1,x_2,\cdots,x_{n+2}),\quad x_i\geqslant x_{i+1},
    \end{equation}
    then $x_i=x_{i+1}=x_{i+2}$ cannot happen; and the number of indices $i$ such that $x_i=x_{i+1}$ should be no more than 2.
\begin{proof}
    Let $\widetilde{\pi}$ be a $\widetilde K$-type in the Dirac cohomology $H_D(X)$. By \cite[Lemma 3.5]{Dong13}, there exists is a highest weight $\mu$ of some $K$-type of $X$, such that the highest weight of $\widetilde \pi$ can be written as  $\{\mu+\rho_c-w\rho\}$, where $\{v\}$ means the unique $\mathfrak k$-dominant weight conjugate to $v$ under the Weyl group $W(\mathfrak{k},\mathfrak{t})$   and $w$ is some element in the Weyl group $W(\mathfrak{g},\mathfrak{t})$. Then by Theorem \ref{thm.HP.condition}, we have
    \begin{equation}\label{Lambda.PRV.component}
        \{\mu+\rho_c-w\rho\}+\rho_c=w\Lambda(X) 
    \end{equation}
for some $w\in W(\mathfrak{g},\mathfrak{t})$. 
    
    The weight $\mu$ is in $\mathbb Z^{n+2}$. Thus $\{ \mu+\rho_c-w\rho \}\in \mathbb Z^{n+2}+\rho_c+\rho$, and hence $\Lambda(X)\in \mathbb Z^{n+2}+\rho$. The equation (\ref{Lambda.PRV.component}) also shows that $\Lambda$ is conjugate to a strongly $\mathfrak k$-dominant weight. Recall a weight $(a_1,a_2,\cdots,a_n|a_{n+1},a_{n+2})$ is strongly $\mathfrak k$-dominant if and only if
    $$a_1>a_2>\cdots> a_n;~a_{n+1}>a_{n+2}.$$
    The action of Weyl group on $\Lambda$ permutes the entries $x_i$ of $\Lambda(X)$. The condition above shows that $\{x_i\}$ should be divided into two parts with size 2 and $n$, and in each part, all entries are pairwise different. Therefore, no entry can appear more than twice. Meanwhile, if $x_i=x_{i+1}$, then they must be assigned to different parts. Notice one part has size 2. Therefore, $x_i=x_{i+1}$ cannot happen more than twice.
\end{proof}
\end{lemma}

As an FS-scattered representation, $X$ is not cohomologically induced from any irreducible module in the weakly good range from any proper $\theta$-stable parabolic subgroups of $G$. In the case of $U(p,q)$, another conjecture of Vogan is recently proved by Wong \cite{Wong24}.

\begin{theorem}\label{Thm.FPP}
    Let $X$ be a unitary irreducible $(\mathfrak g,K)$-module with real infinitesimal character $\Lambda$. If $X$ is not cohomologically induced in the good range from any $(\mathfrak l,L\cap K)$-module $X_L$ for any proper $\theta$-stable parabolic subalgebra $\mathfrak q=\mathfrak l\oplus\mathfrak u$, then one must have $\langle \Lambda,\alpha^\vee \rangle\leqslant 1$ for all simple roots $\alpha\in\Delta^+(\mathfrak g,\mathfrak h)$. In view of (\ref{Lambda.usual.coor}), it means $x_i-x_{i+1}\leqslant 1$.
\end{theorem}

\begin{lemma}\label{Lambda.difference.one}
    Let $X$ be an FS-scattered representation. The infinitesimal character $\Lambda(X)$, as (\ref{Lambda.usual.coor}), should satisfy $x_i-x_{i+1}=0$ or $1$.
    \begin{proof}
By Theorem \ref{Thm.FPP}, $x_i-x_{i+1}\leqslant 1$. By Lemma \ref{lemma.H.P.condition}, $\Lambda(X)$ is integral. Notice that $x_i\geqslant x_{i+1}$. Hence the difference $x_i-x_{i+1}$ is either 0 or 1.
    \end{proof}
\end{lemma}

\begin{lemma}\label{fully.supported.condition}
    Let $G=U(n,2)$ and $X$ be an irreducible $(\mathfrak g,K)$-module corresponding to either a type (a) or a type (b) $\theta$-stable combinatorial datum. Then $X$ is fully supported if and only if all of the following conditions are satisfied:\\
    (1) $\Theta-\nu_\Theta<\Phi+\nu_\Phi;$\\
    (2) At least one of $\Theta+\nu_\Theta$ and $\Phi+\nu_\Phi$ is strictly greater than $\gamma$;\\
    (3) At least one of $\Theta-\nu_\Theta$ and $\Phi-\nu_\Phi$ is strictly less than $-\gamma$.
\begin{proof}
An irreducible module $X$ is fully supported if and only if it cannot be obtained by cohomological induction from the weakly good range. Then the lemma is a special case of lemma \ref{lemma.fully.supported.condition}.
\end{proof}
\end{lemma}

\noindent\textbf{Remark.} An intuitive way to understand the lemma above is to watch the arrows denoting the $\nu_\Theta$ and $\nu_\Phi$. The interior of the arrows, i.e. the intervals $(\Theta-\nu_\Theta, \Theta+\nu_\Theta)$ and $(\Phi-\nu_\Phi,\Phi+\nu_\Phi)$ together should cover all of the contents of $\lambda_a$.

\subsection{Fully supported unitary module candidates}

In this subsection, we search for and count all fully supported unitary modules which satisfy the necessary conditions on the infinitesimal characters stated in Lemma \ref{lemma.H.P.condition} and  Lemma \ref{Lambda.difference.one}. We will discuss the cases when $n>2$. When $n=1$, the considered group is essentially the same as $U(2,1)$; when $n=2$, all fully supported modules we need are shown in example \ref{eg.u22}. In \cite{WongZhang24}, the authors listed all fully supported unitary modules with integral infinitesimal characters and proved that they are all $A_\mathfrak{q}(\lambda)$ modules. Therefore, we will also show the corresponding $\mathfrak q$ and $\lambda$ for the candidates that we find. Moreover, we will count the  number of different $\theta$-stable combinatorial data for each $U(u,2)$. It is important to note that by tensoring a product of the determinant representation, the  data does not change except for adding an integer constant to every block. So, there are a finite number of distinct data shapes but infinitely many representations behind them.

\begin{proposition}\label{prop.fund.1}
    Let $G=U(n,2)$ and $X$ be a fully supported irreducible unitary $(\mathfrak g,K)$-module. Suppose the combinatorial $\theta$-stable datum attached to $X$ is of type (a) and has a size $(1,1)$ parallelogram.
\begin{enumerate}
\item After tensoring a suitable power of the determinant representation, the datum must be
 equivalent to
    
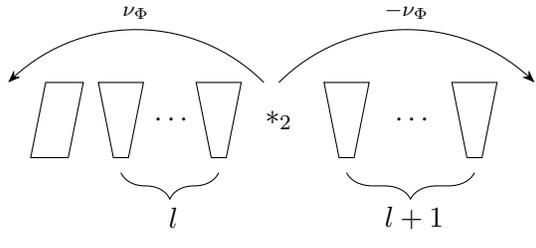
\begin{center}
    \begin{tikzpicture}

        \foreach \y in {4, 5.3, 7, 8.7}
    \draw (\y,1)
    -- (\y+0.2,0)
    -- (\y+0.4,0)
    -- (\y+0.6,1)
    -- cycle;

\foreach \y in {5, 8.2}
    \draw (\y,0.5) node {$\cdots$};

\draw (6.4,0.5) node {$\ast_2$};

\foreach \y in {3.4}
    \draw (\y-0.1,1) -- (\y-0.3, 0) -- (\y+0.2, 0) -- (\y+0.4, 1) -- cycle;

\foreach \x in {4.3}
    \draw decorate [decoration={brace, amplitude=10pt, mirror}]{(\x,-0.2) -- (\x+1.3,-0.2)};

\foreach \x in {7.3}
    \draw decorate [decoration={brace, amplitude=10pt, mirror}]{(\x,-0.2) -- (\x+1.7,-0.2)};
    
\draw (5,-0.8) node {$l$}; \draw (8.2,-0.8) node {$l+1$};
 
\draw[arrows = {-Stealth[]}]    (6.2,1) to [out=135,in=45]node[above]{\scriptsize $\nu_\Phi$} (2.8,1);

\draw[arrows = {-Stealth[]}]    (6.4,1) to [out=45,in=135]node[above]{\scriptsize $-\nu_\Phi$} (9.8,1);

\end{tikzpicture},
\end{center}
or a left-right mirror image of it. In total, there are only two such combinatorial 
$\theta$-stable data.
\item The module corresponding to the $\theta$-stable combinatorial datum is an $A_{\mathfrak q}(\lambda)$-module where the Levi subalgebra $\mathfrak l_0\cong \mathfrak u(0,1)\oplus \mathfrak u(n,1) $ and $\lambda=(-1;1,1,\cdots,1)$. For the left-right mirror image of the datum in (1), it corresponds to an $A_\mathfrak{q}(\lambda)$ where $\mathfrak l_0\cong \mathfrak u(n,1)\oplus \mathfrak u(0,1)$ and $\lambda=(-1,\cdots,-1;1)$.
\end{enumerate}
\begin{proof}
When the $\ast_1$ position is a $(1,1)$-parallelogram, we have $\nu_\Theta=0$ and $0\leqslant \nu_\Phi \leqslant \min \{ l+1,m \}+\frac{r}{2}$, where the size of the block at $\ast_2$ is $(r,1)$. Suppose the datum gives rise to a fully supported unitary module, then lemma \ref{fully.supported.condition} implies that $k=0$, $l+1=m$ and $\nu_\Phi=l+1+\frac{r}{2}$. When the $\ast_2$ position is a $(1,1)$-parallelogram, the datum is just the left-right mirrored image.

Part (2) is from \cite[Proposition 6.2]{WongZhang24}.
\end{proof}
\end{proposition}

For the rest of the data in type (a) and type (b), our discussion based on example \ref{eg.u(n2).fundamental.cases} and lemma \ref{fully.supported.condition}. We start by discussing the unitary representations inside the fundamental rectangle, and then discuss those outside the fundamental rectangle.

\begin{lemma}
     Let $X$ be an irreducible representation of $U(n,2)$ whose combinatorial $\theta$-stable datum is in example \ref{eg.u(n2).fundamental.cases} and  has no parallelogram. Assume $X$ has infinitesimal character in $\rho+\mathbb Z^{n+2}$. Suppose that $\nu$ is in the fundamental rectangle. If $X$ is fully supported, then $\nu$ must be on the boundary of the fundamental rectangle.
\begin{proof}
An observation of the data is that $2\gamma+3=n$. We start with conditions (2) and (3) of lemma \ref{fully.supported.condition}.

Assume $\Theta+\nu_\Theta > \gamma$. Since $X$ has integral infinitesimal character, we have $\Theta+\nu_\Theta \geqslant \gamma+1$. After simplification, we have $\nu_\Theta\geqslant \frac{n-1}{2}-\Theta$. Combining it with the fundamental rectangle inequality $0\leqslant \nu_\Theta\leqslant \frac{n-1}{2}-|\Theta|$, we can deduce that $\Theta\geqslant 0$ and $\nu_\Theta=\frac{n-1}{2}-\Theta$.

By similar arguments, we can summarize and have the following list:
\begin{itemize}
    \item $\Theta+\nu_\Theta > \gamma \Rightarrow$ $\Theta\geqslant 0$ and $\nu_\Theta=\frac{n-1}{2}-\Theta$;
    \item $\Theta-\nu_\Theta <-\gamma \Rightarrow$ $\Theta\leqslant 0$ and $\nu_\Theta=\frac{n-1}{2}+\Theta$;
    \item $\Phi+\nu_\Phi > \gamma \Rightarrow$ $\Phi\geqslant 0$ and $\nu_\Phi=\frac{n-1}{2}-\Phi$;
    \item $\Phi-\nu_\Phi <-\gamma \Rightarrow$ $\Phi\leqslant 0$ and $\nu_\Phi=\frac{n-1}{2}+\Phi$;
\end{itemize}
Every case forces $(\nu_\Theta,\nu_\Phi)$ to stay on the boundary of the fundamental rectangle.
     \end{proof}
 \end{lemma}

The lemma above implies that there are three possible cases:
\begin{itemize}
    \item $\Theta\geqslant 0\geqslant \Phi$, $\Theta+\nu_\Theta=\gamma+1$, and $\Phi-\nu_\Phi=-\gamma-1$;
    \item $\Theta=0$, $\nu_\Theta=\gamma+1$;
    \item $\Phi=0$, $\nu_\Phi=\gamma+1$.
\end{itemize}

The last two cases are essentially the same, and we can put all representations from type (b) into these cases. We use two propositions to describe all the cases mentioned above.

\begin{proposition}\label{prop.fund.2}
    Let $G=U(n,2)$ and $X$ be a fully supported irreducible unitary $(\mathfrak g,K)$-module with infinitesimal character in $\rho+\mathbb Z^{n+2}$. Suppose the combinatorial $\theta$-stable datum attached to $X$ is of type (a) and each $\ast_i$ position is an $(r_i,1)$-rectangle/trapezoid, $r_i=1,2$. Assume $\Theta\geqslant 0\geqslant \Phi$ and $(\nu_\Theta,\nu_\Phi)=(\gamma+1-\Theta,\gamma+1+\Phi)$. Then we have the following facts:
    \begin{enumerate}
\item The combinatorial $\theta$-stable datum is
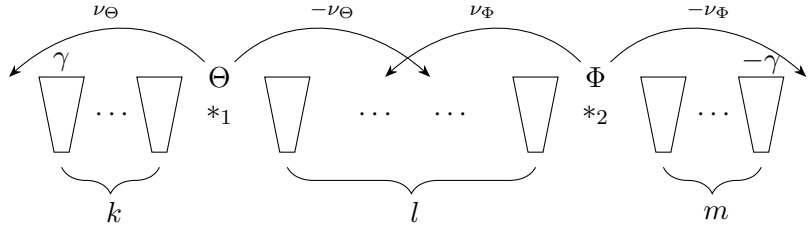
\begin{center}
    \begin{tikzpicture}

\foreach \y in {1, 2.3, 4, 7.3, 9, 10.3}
    \draw (\y,1)
    -- (\y+0.2,0)
    -- (\y+0.4,0)
    -- (\y+0.6,1)
    -- cycle;

\foreach \y in {2, 5.5, 6.5, 10}
    \draw (\y,0.5) node {$\cdots$};

\draw (3.4,0.5) node {$\ast_1$};  \draw (8.4,0.5) node {$\ast_2$};
\draw (3.4,1) node {$\Theta$};  \draw (8.4,1) node {$\Phi$};
\draw (1.3,1.2) node {$\gamma$};
\draw (10.6,1.2) node {$-\gamma$};

\foreach \x in {1.3, 9.3}
    \draw decorate [decoration={brace, amplitude=10pt, mirror}]{(\x,-0.2) -- (\x+1.3,-0.2)};
\foreach \x in {4.3}
    \draw decorate [decoration={brace, amplitude=10pt, mirror}]{(\x,-0.2) -- (\x+3.3,-0.2)};
    
\draw (2,-0.8) node {$k$}; \draw (6,-0.8) node {$l$}; \draw (10,-0.8) node {$m$};

\draw[arrows = {-Stealth[]}] (3.2,1.2) to [out=135,in=45]node[above]{\scriptsize $\nu_\Theta$} (0.6,1);
\draw[arrows = {-Stealth[]}] (3.6,1.2) to [out=45,in=135]node[above]{\scriptsize $-\nu_\Theta$}(6.2,1);

\draw[arrows = {-Stealth[]}] (8.2,1.2) to [out=135,in=45]node[above]{\scriptsize $\nu_\Phi$} (5.6,1);
\draw[arrows = {-Stealth[]}] (8.6,1.2) to [out=45,in=135]node[above]{\scriptsize $-\nu_\Phi$}(11.2,1);

\end{tikzpicture}
\end{center}
In the datum, $(\Phi+\nu_\Phi)-(\Theta-\nu_\Theta)=1$. In other words, the difference between ends of the two arrows is $1$.
\item There are $n-1$ many different such combinatorial $\theta$-stable data.
\item The corresponding module is an $A_\mathfrak{q}(\lambda)$, where $\mathfrak l_0\cong \mathfrak u(n-2\Theta-1,1) \oplus \mathfrak u (n+2\Phi-1,1)$ and $\lambda=(-1,\cdots,-1;1,\cdots,1)$.
    \end{enumerate}
\begin{proof}
    The positions of the arrows on the left and right sides are determined by the value of $(\nu_\Theta,\nu_\Phi)$. The arrows in the middle intersect with each other because of part (1) of Lemma \ref{fully.supported.condition}, which is $\nu_\Theta+\nu_\Phi>\Theta-\Phi$. The coordinate $(\nu_\Theta,\nu_\Phi)=(\gamma+1-\Theta,\gamma+1+\Phi)$ is on the vertex of the fundamental rectangle. Recall that in Example \ref{eg.u(n2).fundamental.cases}, $(\nu_\Theta,\nu_\Phi)$ should satisfy one of a list of conditions to make $X$ unitary. In fact, when $(\nu_\Theta,\nu_\Phi)$ is the vertex, the only possible condition it can satisfy is $\nu_\Theta+\nu_\Phi\leqslant\Theta-\Phi+1$.
    
    Define $\tilde k:=k+\frac{r_1}{2}$, $\tilde m:=m+\frac{r_2}{2}$, and $\tilde l:= l-1+\frac{r_1+r_2}{2}$. Then, regardless of the value of $r_i$, one has $\tilde l=\Theta-\Phi$, and $\nu_\Theta+\nu_\Phi=2(\gamma+1)-\tilde l$. Substituting the value of $\nu_\Theta+\nu_\Phi$ into the inequalities from the previous paragraph, then we have
$$\tilde l<2(\gamma+1)-\tilde l\leqslant \tilde l+1.$$
    Because $l$ is an integer, we must have $2(\gamma+1)-\tilde l=\tilde l+1$. Then
    $$(\Phi+\nu_\Phi)-(\Theta-\nu_\Theta)=(\nu_\Theta+\nu_\Phi)-(\Theta-\Phi)=(2(\gamma+1)-\tilde l)-\tilde l=1.$$
    The proof of part (1) is complete.

    Notice that $2\gamma+3=n=\tilde k+\tilde l+\tilde m+1$. Combining the equation $2(\gamma+1)-\tilde l=\tilde l+1$, we have $\tilde k+\tilde m=\frac{n}{2}$. Because $\tilde k=k+\frac{r_1}{2}$, there are $n-1$ many choices for $\tilde k$. Each of the choices gives a different combinatorial $\theta$-stable datum. The proof of part (2) is done.

    In  \cite[Section 5]{WongZhang24}, the authors have proved that these modules are $A_\mathfrak{q}(\lambda)$-modules with $\mathfrak l_0=\mathfrak{u}(n-2\Theta-1,1)\oplus \mathfrak{u}(n+2\Phi-1,1)$. It is easy to figure out the value of $\lambda$ by $\lambda+\rho=\Lambda$. The proof of part (3) is complete.
\end{proof}
\end{proposition}

\begin{proposition}\label{prop.fund.3}
    Let $G=U(n,2)$ and $X$ be a fully supported irreducible unitary $(\mathfrak g,K)$-module with infinitesimal character in $\rho+\mathbb Z^{n+2}$. Suppose the combinatorial $\theta$-stable datum attached to $X$ is one of the following cases.
\begin{itemize}
    \item It is of type (a) and each $\ast_i$ position is an $(r_i,1)$-rectangle/trapezoid; And $\Theta= 0\geqslant \Phi$ and $\nu_\Theta=\gamma+1$.
    \item It is of type (a) and each $\ast_i$ position is an $(r_i,1)$-rectangle/trapezoid; And $\Theta \geqslant 0 = \Phi$ and $\nu_\Phi=\gamma+1$.
    \item It is of type (b) with a $0$-block in the middle and $\nu_\Theta=\gamma+1$.
\end{itemize}
   Then we have the following facts:
\begin{enumerate}
\item In the combinatorial $\theta$-stable datum, the pair of the arrows reaching both sides, i.e. $\pm(\gamma+1)$, come from the $0$-block as shown in the picture; and the other pair of arrows is restricted inside $[-\gamma,\gamma]$. 
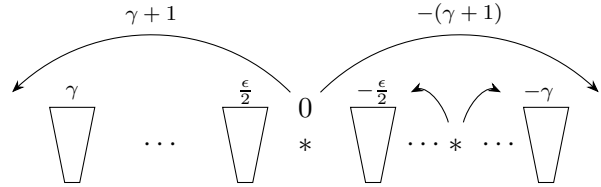
\begin{center}
    \begin{tikzpicture}
\foreach \y in {0, 2.3, 4, 6.3}
    \draw (\y,1)
    -- (\y+0.2,0)
    -- (\y+0.4,0)
    -- (\y+0.6,1)
    -- cycle;

\foreach \y in {1.5, 5, 6}
    \draw (\y,0.5) node {$\cdots$};

\draw (3.4,0.5) node {$\ast$};
\draw (3.4,1) node {$0$};
\draw (0.3,1.2) node {\scriptsize $\gamma$};
\draw (6.5,1.2) node {\scriptsize $-\gamma$};

\draw (5.4,0.5) node {$\ast$};
\draw[arrows = {-Stealth[]}] (5.3,0.8) to [out=105,in=45]node[above]{} (4.8,1.2);
\draw[arrows = {-Stealth[]}] (5.5,0.8) to [out=75,in=135]node[above]{} (6.0,1.2);



\draw[arrows = {-Stealth[]}] (3.2,1.2) to [out=135,in=45]node[above]{\scriptsize $\gamma+1$} (-0.5,1.2);
\draw[arrows = {-Stealth[]}] (3.6,1.2) to [out=45,in=135]node[above]{\scriptsize $-(\gamma+1)$} (7.3,1.2);
\draw (2.6,1.2) node {\scriptsize $\frac{\epsilon}{2}$};
\draw (4.3,1.2) node {\scriptsize $-\frac{\epsilon}{2}$};

\end{tikzpicture}
\end{center}
In the picture, $\epsilon=1$ or $2$ and $\epsilon\equiv n+1 ~(\emph{mod}~2)$. When $n$ is even, the $0$-block in the middle is a rectangle of size $(1,1)$ if the datum is of type (a), or of size $(2,2)$ if the datum is of type (b). When $n$ is odd, the $0$-block in the middle is a trapezoid of size $(2,1)$ if the datum is of type (a), or of size $(3,2)$ if the datum is of type (b).
\item If we require the corresponding modules satisfy the H.P.-condition, all three cases together gives $\frac{(n-2)(n-3)}{2}$ different combinatorial $\theta$-stable data.
\item The corresponding module is an $A_\mathfrak{q}(\lambda)$. When $\Phi\geqslant 0=\Theta$, the Levi subalgebra $\mathfrak l_0\cong \mathfrak u(n-1,1)\oplus\mathfrak u(0,1) \oplus\mathfrak u(1,0)$, and $\lambda=(-1,-1,\cdots,-1;\Phi+\nu_\Phi+\frac{n-1}{2};\Phi-\nu_\Phi+\frac{n+1}{2})$. When $\Theta\geqslant 0=\Phi$, the Levi subalgebra $\mathfrak l_0\cong \mathfrak u(1,0) \oplus\mathfrak u(0,1) \oplus\mathfrak u(n-1,1)$, and $\lambda=(\Theta+\nu_\Theta-\frac{n+1}{2};\Theta-\nu_\Theta-\frac{n-1}{2};1,1,\cdots,1)$.
\end{enumerate}
\begin{proof}
In part (1), the statement about arrows is trivial. Recall $2\gamma+3=n$, and that $\gamma-\frac{\epsilon}{2}\in\mathbb Z$. When $n$ is even, $\gamma\in \mathbb Z+\frac{1}{2}$, and hence $\epsilon=1$. The $0$-block has to be a rectangle due to definition \ref{def.gamma.block}. A similar argument applies when $n$ is odd.

For part (2), assume the other pair of arrows comes out from the $x$-block. The datum is of type (b) when $x=0$. Otherwise, it is of type (a). Say these arrows denote $\nu_x$. The value of $x$ could be any integer and half integer in the interval $(\gamma,\gamma)$. With a fixed $x$-block, for example when $x\geqslant 0$, the arrow for $x+\nu_x$ could end at any block of size $(1,0)$ between $\gamma$ and $x$. The H.P.-condition (Lemma \ref{lemma.H.P.condition}) works when the $x$-block is a trapezoid. In such case, it wipes out the choice $\nu_x=0$. Otherwise, the value $x$ will show up three times in the infinitesimal character. Therefore, $\nu_x$ could be either $\{1,2,\cdots\}$ or $\{\frac{1}{2},\frac{3}{2},\frac{5}{2},\cdots\}$ until $x+\nu_x$ hits $\gamma$ or $x-\nu_x$ hits $-\gamma$. After all conditions are set, it is not hard to find a pattern to count all different $\theta$-stable data.

Part (3) is  complete in \cite[Section 5]{WongZhang24}.
\end{proof}
\end{proposition}

Now, we use a proposition to explain the fully supported, unitary representations when $(\nu_\Theta,\nu_\Phi)$ is outside the fundamental rectangle.

\begin{proposition}\label{prop.fund.4}

Let $G=U(n,2)$ and $X$ be a fully supported irreducible unitary $(\mathfrak g,K)$-module with infinitesimal character in $\rho+\mathbb Z^{n+2}$. Assume $(\nu_\Theta,\nu_\Phi)$ is outside the fundamental rectangle. Then we have the following facts:
 
\begin{enumerate}
\item The combinatorial $\theta$-stable datum of $X$ is either 
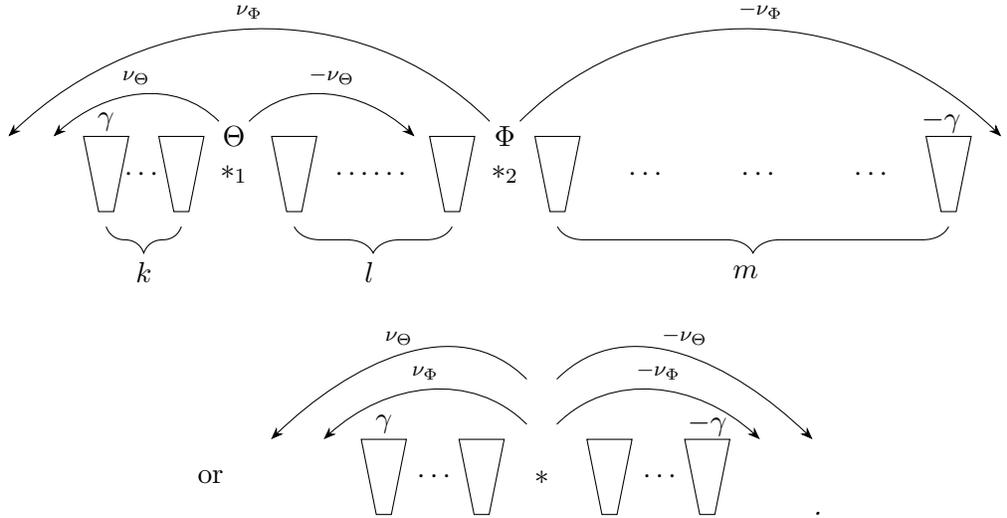
\begin{center}
    \begin{tikzpicture}

\foreach \y in {1, 2, 3.5, 5.6, 7.0, 12.2}
    \draw (\y,1)
    -- (\y+0.2,0)
    -- (\y+0.4,0)
    -- (\y+0.6,1)
    -- cycle;

\foreach \y in {1.8, 4.6, 5.1, 8.5, 10, 11.5}
    \draw (\y,0.5) node {$\cdots$};

\draw (3,0.5) node {$\ast_1$};  \draw (6.6,0.5) node {$\ast_2$};
\draw (3.0,1) node {$\Theta$};  \draw (6.6,1) node {$\Phi$};
\draw (1.3,1.2) node {$\gamma$};
\draw (12.4,1.2) node {$-\gamma$};

\foreach \x in {1.3}
    \draw decorate [decoration={brace, amplitude=10pt, mirror}]{(\x,-0.2) -- (\x+1,-0.2)};
\foreach \x in {3.8}
    \draw decorate [decoration={brace, amplitude=10pt, mirror}]{(\x,-0.2) -- (\x+2.1,-0.2)};
\foreach \x in {7.3}
    \draw decorate [decoration={brace, amplitude=10pt, mirror}]{(\x,-0.2) -- (\x+5.2,-0.2)};
    
\draw (1.8,-0.8) node {$k$}; \draw (4.8,-0.8) node {$l$}; \draw (9.8,-0.8) node {$m$};

\draw[arrows = {-Stealth[]}] (2.8,1.2) to [out=135,in=45]node[above]{\scriptsize $\nu_\Theta$} (0.6,1);
\draw[arrows = {-Stealth[]}] (3.2,1.2) to [out=45,in=135]node[above]{\scriptsize $-\nu_\Theta$}(5.4,1);

\draw[arrows = {-Stealth[]}] (6.4,1.2) to [out=135,in=45]node[above]{\scriptsize $\nu_\Phi$} (0,1.0);
\draw[arrows = {-Stealth[]}] (6.8,1.2) to [out=45,in=135]node[above]{\scriptsize $-\nu_\Phi$}(13.2,1);

\end{tikzpicture}
\end{center}

 \begin{center}
    \begin{tikzpicture}
\draw (-1,0.5) node {\emph{or}};

\foreach \y in {1, 2.3, 4, 5.3}
    \draw (\y,1)
    -- (\y+0.2,0)
    -- (\y+0.4,0)
    -- (\y+0.6,1)
    -- cycle;

\foreach \y in {2,5}
    \draw (\y,0.5) node {$\cdots$};

\draw (3.4,0.5) node {$\ast$};

\draw (1.3,1.2) node {$\gamma$};  
\draw (5.6,1.2) node {$-\gamma$};   
\draw[arrows = {-Stealth[]}] (3.2,1.2) to [out=135,in=45]node[above]{\scriptsize $\nu_\Phi$}(0.5,1);
\draw[arrows = {-Stealth[]}] (3.6,1.2) to [out=45,in=135]node[above]{\scriptsize $-\nu_\Phi$}(6.3,1);
\draw[arrows = {-Stealth[]}] (3.2,1.8) to [out=135,in=45]node[above]{\scriptsize $\nu_\Theta$}(-0.2,1);
\draw[arrows = {-Stealth[]}] (3.6,1.8) to [out=45,in=135]node[above]{\scriptsize $-\nu_\Theta$}(7,1);

\end{tikzpicture}.
\end{center}
In the first picture, each $\ast_i$ position is an $(r_i,1)$-rectangle/trapezoid with $r_i=1,2$; and in the second picture, the $\ast$ position is a $(2,2)$-rectangle or a $(3,2)$-trapezoid. And the datum of $X$ could also be left-right mirrored image of the first picture.

\item If we require the corresponding modules to satisfy the H.P.-condition, there are $2n-3$ different combinatorial $\theta$-stable data.
\item The corresponding module is an $A_\mathfrak{q}(\lambda)$. When $\Theta\geqslant\Phi>0$, the Levi subalgebra is $\mathfrak l_0\cong \mathfrak u(n-1,2) \oplus \mathfrak u (1,0)$ and $\lambda=(0,0,\cdots,0; 2\Theta+1)$; when $0>\Theta\geqslant\Phi$, $\mathfrak l_0\cong \mathfrak u(1,0) \oplus \mathfrak u (n-1,2)$ and $\lambda=(2\Phi-1;0,0,\cdots,0)$. When $\Theta=\Phi=0$, it corresponds to the trivial module.
\end{enumerate}

\begin{proof}
According to the example \ref{eg.u(n2).fundamental.cases}, $(\nu_\Theta,\nu_\Phi)=(\frac{n-1}{2}-|\Theta|, \frac{n+1}{2}-|\Phi|)$ or $(\nu_\Theta,\nu_\Phi)=(\frac{n+1}{2}-|\Theta|, \frac{n-1}{2}-|\Phi|)$. When $\Theta>\Phi>0$, $\Phi-\nu_\Phi$ has to be $-\gamma-1$ because $X$ is fully supported. Now, we have the first picture. When $0>\Theta>\Phi$, we have the left-right mirrored image of the first picture. When $\Theta=\Phi$, we may assume $\nu_\Theta\geqslant\nu_\Phi$. At least one of $|\Theta\pm \nu_\Theta|$ equals $|\gamma+2|$. Notice $\nu_\Theta-\nu_\Phi=1$, then we have the second picture. The proof of part (1) is complete.

Assume $(\nu_\Theta,\nu_\Phi)=(\frac{n-1}{2}-|\Theta|, \frac{n+1}{2}-|\Phi|)$ and $ \Theta \geqslant \Phi>0$. One has $\Phi-\nu_\Phi=-\gamma-1$ and $\Phi+\nu_\Phi=\gamma+2$. Hence $\Phi=\frac{1}{2}$. Then for each $\Theta$, there is only one such datum because $(\nu_\Theta,\nu_\Phi)$ are fixed. All possible value of $\Theta$ is $\frac{1}{2},1,\frac{3}{2},\cdots,\gamma+\frac{1}{2}$. Notice that $2\gamma+3=n$. There are $n-2$ different data. The left-right mirrored images produce same amount of data. One special case is when $\Theta=\Phi=0$. In this case, $X$ is the trivial representation. Therefore, in total there are $2n-3$ different data. The proof of part (2) is complete.

Part (3) is complete in \cite[Section 5]{WongZhang24}.

\end{proof}

\end{proposition}

\subsection{The Dirac series of $U(n,2)$}

A tricky way to prove a representation has nonzero Dirac cohomology is to prove that it has nonzero Dirac index. Let us recall the definition of Dirac index, see \cite[Section 4]{DongWong21} for more details.

We know that the Dirac cohomology of the $(\mathfrak{g},K)$-module $X$ is defined as the Kernel of $\mathrm{ker}(D)/\big(\mathrm{ker}(D)\cap \mathrm{im}(D)\big)$ in $X\otimes S(\mathfrak{p})$. When $G$ and $K$ are of equal rank, the spin module $\mathrm{Spin}_G$ can be decomposed into odd part $\mathrm{Spin}_G^-$ and even part $\mathrm{Spin}_G^+$, and $D$ maps $X\otimes \mathrm{Spin}_G^{\pm}$ to $\mathrm{Spin}_G^{\mp}$. Hence, $H_D(X)$ breaks up into the odd part and the even part, which are denoted by $H_D^-(X)$ and
$H_D^+(X)$. Then the Dirac index is defined as 
\[\mathrm{DI}(X):=H_D^+(X)-H_D^-(X)\]
in the Grothendieck group of $\widetilde{K}$-modules.

The Dirac index is easier to compute than the Dirac cohomology. It was found by \cite{MPV17} that there is a translation principle on the Dirac index. By such a principle, \cite{DongWong21} got the Dirac index of $A_{\mathfrak{q}}(\lambda)$:
\[\mathrm{DI}\big(A_{\mathfrak{q}}(\lambda)\big)=\sum_{w\in W^1(\mathfrak{l},\mathfrak{t})}\det(w)\widetilde{E}_{w(\lambda+\rho)},\]
where $W^1(\mathfrak{l},\mathfrak{t})=W(\mathfrak{l},\mathfrak{t})\cap W^1$ with 
\[W^1:=\{w\in W(\mathfrak{g},\mathfrak{t})|w\rho~{\rm is }~ \Delta^{+}(\mathfrak{k},\mathfrak{t}){\rm-dominant}\},\] 
and $\widetilde{E}_{\mu}$ is defined by 
\[\widetilde{E}_{\mu}=\begin{cases} \det(w')E_{\{\mu\}-\rho_c}; &\ \mu \ \text{is}\ \mathfrak{k}\text{-regular},\  w'\in W(\mathfrak{k},\mathfrak{t})\ \text{s.t.}\ w'\mu\ \text{is}\ \Delta^+(\mathfrak{k},\mathfrak{t})\text{-dominant},\\ 0; & \ \text{otherwise}.\end{cases}
\]
We mention that such result not only holds for weakly fair range but also for mediocre range, see definition of mediocre range in \cite[Definition 3.4]{Trapa01}. By \cite[Theorem 5.11] {DongWong21}, In the case of $A_\mathfrak{q}(\lambda)$ modules of $U(p,q)$, there is a necessary and sufficient condition to determine the non-vanishing of the Dirac index.

\begin{theorem}\cite[Lemma 5.6]{DongWong21}\label{thm.dirac.index.nonzero}
    Let $X=A_\mathfrak{q}(\lambda)$ be a module of $U(p,q)$, where the Levi subalgebra $\mathfrak l_0$ of $\mathfrak q$ is $\mathfrak u(p_1,q_1)\oplus \cdots \oplus \mathfrak{u}(p_r,q_r)$ and
$$\lambda=(\underbrace{\lambda_1,\cdots,\lambda_1}_{p_1+q_1};\cdots;\underbrace{\lambda_r,\cdots,\lambda_r}_{p_r+q_r}).$$
Write the infinitesimal character $\Lambda$ of $A_\mathfrak{q}(\lambda)$ as
$$\Lambda=(\nu^{(1)}_1,\cdots,\nu^{(1)}_{p_1+q_1};\cdots;\nu^{(r)}_1,\cdots,\nu^{(r)}_{p_r+q_r})=\lambda+\rho.$$
Define $R_{ij}={\rm Card}\{ \nu^{(i)}_1,\cdots,\nu^{(i)}_{p_i+q_i} \} \cap \{ \nu^{(j)}_1,\cdots,\nu^{(j)}_{p_j+q_j} \}$. Then $X$ has nonzero Dirac index if and only if the following set of equations has non-negative integer solutions $(a_{ij},b_{ij})$.
\begin{equation}
    \begin{cases}
a_{ij}+b_{ij}=R_{ij},\\
 \sum_{s>m}a_{ms}+\sum_{t<m} b_{tm}\leqslant p_m,\\
 \sum_{t<m}a_{tm}+\sum_{s>m} b_{ms}\leqslant q_m.
    \end{cases}
\end{equation}
\end{theorem}

\begin{theorem}\label{thm.scattered.reps}
    After tensoring a power of the determinant representation, all FS-scattered representations are those $A_\mathfrak{q}(\lambda)$ listed in proposition \ref{prop.fund.1}, \ref{prop.fund.2}, \ref{prop.fund.3}, \ref{prop.fund.4} and example \ref{eg.u22}. For each $n$, there are $\frac{n^2+n+2}{2}$ many such FS-scattered representations.
    \begin{proof}
 An FS-scattered representation is a fully supported unitary representation with nonzero Dirac cohomology. We have summarized all fully supported unitary representations that satisfy the H.P.-condition in Proposition \ref{prop.fund.1}, \ref{prop.fund.2}, \ref{prop.fund.3} and \ref{prop.fund.4}. In fact, all of them have nonzero Dirac index, hence nonzero Dirac cohomology.

Using Theorem \ref{thm.dirac.index.nonzero}, we check briefly the case in Proposition \ref{prop.fund.1}. 
In this case, $R_{12}=2$ $p_1,~p_2$ are both greater than 2, and $q_1=q_2=1$. Then $a_{12}=b_{12}=1$ is the unique non-negative integer solution.

We omit the similar argument of the cases in Proposition \ref{prop.fund.2}, \ref{prop.fund.3} and \ref{prop.fund.4}.

The number of different combinatorial $\theta$-stable data are counted in each proposition listed above. In total there are $\frac{n^2+n+2}{2}$ different data.
    \end{proof}
\end{theorem}

For all string representations in the Dirac series of $U(n,2)$, they are all obtained by cohomological induction in the weakly good range from FS-scattered representations of $U(p,q)$ in smaller rank with $q\leqslant 2$.

\begin{theorem}\label{thm.string.reps}
In the Dirac series of $U(n,2)$, all FS-scattered representations are listed in theorem \ref{thm.scattered.reps}. For a string representation $X$ in the Dirac series of $U(n,2)$, there exists a proper $\theta$-stable parabolic $\mathfrak q=\mathfrak l\oplus\mathfrak u$ and a $(\mathfrak l, L\cap K)$-module $\pi$, such that
\begin{enumerate}
    \item $X$ is cohomologically induced from $\pi$ in the weakly good range;
    \item $\mathfrak l_0=\mathfrak u(n_1,r_1)\oplus\cdots\oplus \mathfrak u(n_m,r_m)$;
    \item $\pi$ restricted on each Levi factor $\mathfrak u(n_i,r_i)$, denoted by $\pi_i$, has nonzero Dirac cohomology;
    \item if $r_i=2$ and $n_i>1$, $\pi_i$ is then a fully supported representation listed in theorem \ref{thm.scattered.reps} where the $n_i$ is less than $n$; otherwise, $\pi_i$ is a finite dimensional representation.
\end{enumerate}
\begin{proof}
Let $X$ be a representation of $U(n,2)$ in the Dirac series and suppose its combinatorial $\theta$-stable datum $\mathcal D$ is $(\lambda_a(\delta),\nu)$ where $\delta$ is its lowest $K$-type. If $X$ is fully supported, it is easy to see that $\mathcal D$ has to be fundamental. Hence, it must be one of the cases listed in theorem \ref{thm.scattered.reps}.

Now, suppose $X$ is in the string representations, which means it has to be cohomologically induced from some $\pi$ of a proper $\theta$-stable parabolic $\mathfrak q$. Partition the datum $\mathcal D$ in a way such that each component $\mathcal D_i$ is one of the following cases:
\begin{itemize}
    \item a collection of consecutive blocks of size $(1,0)$;
    \item a collection of consecutive blocks of size $(0,1)$;
    \item a datum corresponding to a trivial representation of $U(p,1)$ with $(p<n)$ or $U(1,2)$;
    \item a datum corresponding to a fully supported representation of $U(p,2)$ with $(p<n)$.
\end{itemize}
In fact, we can always partition the datum into fundamental pieces. For the sake of convenience, we may put consecutive blocks of size $(1,0)$ (resp. $(0,1)$) together. Locally, they correspond to finite-dimensional irreducible representations of compact unitary groups. Moreover, after a quick observation, we may assume the rest pieces of the partitions correspond to fully supported representations. Theorem \ref{thm.weakly.good.ind.unitary} implies that each datum $\mathcal D_i$ corresponds to a unitary representation $\pi_i$. So, if $\mathcal D_i$ corresponds to a representation of $U(p,1)$ or $U(1,2)$, it has to be the trivial representation because it is unitary and fully supported.

A key fact is that the corresponding segments $[e_i,b_i]$ can be made to satisfy $b_i\geqslant e_{i+1}$. Then lemma \ref{lemma.fully.supported.condition} implies that the cohomological induction constructed from the partition $\mathcal D=\coprod_i \mathcal D_i$ is in the weakly good range. Because $X$ is in the Dirac series, corollary \ref{cor.HD.induction} forces that $\pi$ has nonzero Dirac cohomology. As a consequence, each component $\pi_i$ has to have nonzero Dirac cohomology. Therefore, if $\mathcal D_i$ corresponds to a representation of $U(p,2)$, it has to be an FS-scattered representation of $U(p,2)$.
    \end{proof}
\end{theorem}

\noindent\textbf{Remark.} In the proof of the above theorem, a tricky part is partitioning the datum into unions of size $(1,0)$ blocks, unions of size $(0,1)$ blocks and fully supported components as $\mathcal D=\coprod_i \mathcal D_i$. And we say the corresponding segments $[e_i,b_i]$ satisfy $b_i\geqslant e_{i+1}$. This trick can work out because our group is $U(n,2)$. If one goes back to the fundamental data which correspond to fully supported representations, the only case one should worry about is in proposition \ref{prop.fund.4}. To be precise, the partition could locally look like the left picture where the two arrows are included in $\mathcal D_{i+1}$. This means $\mathcal D_{i+1}$ corresponds to a fully supported representation of some $U(p,2)$. Because the group we have is $U(n,2)$, the rest of blocks outside $\mathcal D_{i+1}$ in the whole datum have to be of size $(1,0)$. We see that in the left picture $e_i+1=b_{i+1}$, which is not what we want. But we can slightly adjust the partition as shown in the right picture. In the new partition, we still have the two arrows included in $\mathcal D^\prime_{i+1}$. But we move one block from $\mathcal D_i$ to $\mathcal D^\prime_{i+1}$. Then in the new partition, the segments have $e^\prime_i=b_{i+1}^\prime$. The issue is fixed.

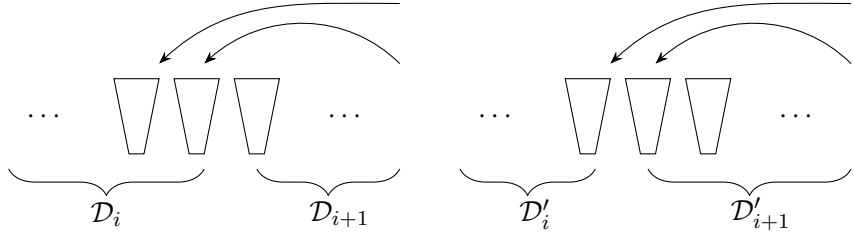
\begin{center}
    \begin{tikzpicture}

\foreach \y in {1, 0.2, -0.6}
    \draw (\y,1)
    -- (\y+0.2,0)
    -- (\y+0.4,0)
    -- (\y+0.6,1)
    -- cycle;

\foreach \y in {-1.5, 2.5}
    \draw (\y,0.5) node {$\cdots$};

\draw decorate [decoration={brace, amplitude=10pt, mirror}]{(1.3,-0.2) -- (3.2,-0.2)};
\draw decorate [decoration={brace, amplitude=10pt, mirror}]{(-2,-0.2) -- (0.6,-0.2)};

\draw (2.4,-0.8) node {$\mathcal D_{i+1}$};
\draw (-0.7,-0.8) node {$\mathcal D_{i}$};

\draw[arrows = {-Stealth[]}] (3.2,1.2) to [out=135,in=45] (0.6,1.2);

\draw[arrows = {-Stealth[]}] (3.2, 2) to [out=180,in=45] (0,1.2);

\foreach \y in {7, 6.2, 5.4}
    \draw (\y,1)
    -- (\y+0.2,0)
    -- (\y+0.4,0)
    -- (\y+0.6,1)
    -- cycle;

\foreach \y in {4.5, 8.5}
    \draw (\y,0.5) node {$\cdots$};

\draw decorate [decoration={brace, amplitude=10pt, mirror}]{(6.5,-0.2) -- (9.2,-0.2)};
\draw decorate [decoration={brace, amplitude=10pt, mirror}]{(4,-0.2) -- (5.8,-0.2)};

\draw (8,-0.8) node {$\mathcal D_{i+1}^\prime$};
\draw (5,-0.8) node {$\mathcal D_{i}^\prime$};

\draw[arrows = {-Stealth[]}] (9.2,1.2) to [out=135,in=45] (6.6,1.2);

\draw[arrows = {-Stealth[]}] (9.2, 2) to [out=180,in=45] (6,1.2);

\end{tikzpicture}
\end{center}
-0

But in general cases when the group is $U(p,q)$, the fundamental data are more complicated. This trick cannot be reapplied in the general case.

\section{Dirac cohomology and spin lowest $K$-type}

To compute the Dirac cohomology of a unitary $(\mathfrak{g},K)$-module $X$ effectively, \cite{Dong13} defines \emph{the spin lowest $K$-type} of $X$.  Let us briefly summarize some results of \cite{Dong13}.

Let $W^1:=\{w\in W(\mathfrak{g},\mathfrak{t})|w\rho~{\rm is }~ \Delta^{+}(\mathfrak{k},\mathfrak{t}){\rm-dominant}\}$, and let 
\[\{\rho_n^{(j)}|j=1,\dots,m\}:=\{w\rho-\rho_c|w\in W^1\}.\] 

Then a $K$-type $\sigma$ of a unitary $(\mathfrak{g},K)$-module $X$ with infinitesimal character $\Lambda$ is called spin lowest $K$-type if 
\begin{equation}\label{spin_lowest}w\Lambda-\rho_c=\{\sigma-\rho_n^{(j)}\}\end{equation}
for some $w\in W(\mathfrak{g},\mathfrak{t})$ and some $\rho_n^{(j)}$, where $\{\bullet\}$ denotes the $\Delta^{+}(\mathfrak{k},\mathfrak{t})$-dominant element to which the element inside the bracket is conjugated via $W(\mathfrak{k},\mathfrak{t})$. Moreover, the Dirac cohomology of $X$ 
\[H_D(X)\simeq \mathop{\bigoplus}\limits_{\tiny \begin{array}{c}\sigma\in \{\text{spin lowest}\ K\text{-types}\}\\ \rho_n^{(j)}\ \text{satisfies}\ \eqref{spin_lowest}\end{array}} 2^{[\frac{l_0}{2}]}E_{\{\sigma-\rho_n^{(j)}\}}
\]
as $\widetilde{K}$-module, where $l_0=\mathrm{rank}(\mathfrak{g})-\mathrm{rank}(\mathfrak{k})$, $E_{\bullet}$ denotes the finite dimensional irreducible  $\widetilde{K}$-module with highest weight $\bullet$, and the sum over $\sigma$ counts the multiplicity of $\sigma$ in $X$. In the case that $G$ is of Hermitian symmetric type, $G$ and $K$ are of equal rank, $l_0=0$. 

From the Theorem \ref{thm.scattered.reps} and Theorem \ref{thm.string.reps}, one has obtained all the Dirac series of $U(n,2)$. We continue computing the Dirac cohomology of them explicitly. Roughly, we will first find out the possible $\widetilde{K}$-types which possibly show up in the Dirac cohomology, and then find the spin lowest $K$-types, from which one can get the Dirac cohomology.

By Theorem \ref{thm.scattered.reps}, all the scattered representations are $A_{\mathfrak{q}}(\lambda)$-modules. We will show the following results for $U(n,2)$.
\begin{theorem}
Given any FS-scattered representation $A_{\mathfrak{q}}(\lambda)$ of $U(n,2)$, then for each $w\in W^1(\mathfrak{l},\mathfrak{t})$ such that $w(\lambda+\rho)$ is $\mathfrak{k}$-regular, $E_{\{w(\lambda+\rho)\}-\rho_c}$ will occur in the Dirac cohomology of $A_{\mathfrak{q}}(\lambda)$. Moreover, there is no cancellation between the even part and the odd part of Dirac cohomology. 

Given a spin lowest $K$-type $\sigma$, there is a subset $A_{\sigma}$ of $\{\rho_n^{(j)}|j=1,\dots,m\}$ such that $\{\sigma-\rho_n^{(i)}\}$ contributes to the Dirac cohomology $\forall \rho_n^{(i)}\in A_{\sigma}$. Then given any different spin lowest $K$-types $\sigma_1$ and $\sigma_2$, one has $A_{\sigma_1}\cap A_{\sigma_2}=\emptyset$. And 
\[\bigsqcup_{\sigma\in \ \{\text{spin lowest}\ K\text{-types}\}}A_{\sigma}\subset \{w\rho-\rho_c|w\in W^1(\mathfrak{l},\mathfrak{t})\}.\]    
\end{theorem}
\begin{proof}
We will begin with listing the Dirac series, their spin lowest $K$-types and the corresponding $\rho_n^{(j)}$, which implies the statements about $A_\sigma$. In the case when $G=U(n,2)$, all of the Dirac series are $A_{\mathfrak q}(\lambda)$ modules. One can recognize an $A_{\mathfrak q}(\lambda)$ module from its Levi subalgebras $\mathfrak l_0$ of $\mathfrak q$ and $\lambda$.

\begin{itemize}
\item[(a)]  $\mathfrak l_0\cong \mathfrak u(0,1)\oplus \mathfrak u(n,1) $ and 
\[\lambda=(-1;\underbrace{1,1,\cdots,1}_{n+1}).\] The spin lowest $K$-types are
\[(0,\dots,0,-1\ |\ n,1),\ \text{and}\ (0,\dots,0\ |\ n-1,1).\]
The first corresponds to only one $\rho_n^{(j)}=(0,\dots,0,-1\ | \ \frac{n}{2},-\frac{n-2}{2})$, and the second corresponds to $n-1$ different $\rho_n^{(j)}$'s:
\[(0,\dots,0,\underbrace{-1,\dots,-1}_{m+1}\ | \  \frac{n}{2},-\frac{n-2-2m}{2}),\ 1\leqslant m\leqslant n-1.\]

\item[(b)] $\mathfrak l_0\cong \mathfrak u(n-2\Theta-1,1) \oplus \mathfrak u (n+2\Phi-1,1)$ with $2\Phi-2\Theta=-n+2, \Phi,\Theta\in \frac{1}{2}\mathbb{Z}$, and 
\[\lambda=(\underbrace{-1,-1,\cdots,-1}_{n-2\Theta};\underbrace{1,1,\cdots,1}_{n+2\Phi}).\] 
The spin lowest $K$-types are
\[(1,0,\dots,0\ |\ n-2\Theta-2,-n-2\Phi-1)\ \text{and}\ (0,\dots,0,-1\ |\ n-2\Theta-1,-n-2\Phi).\]
The first corresponds to $\rho_n^{(j)}=(1,0,\dots,0\ | \ \frac{n-2}{2},-\frac{n}{2})$, and the second corresponds to $\rho_n^{(j)}=(0,\dots,0,-1\ | \  \frac{n}{2},-\frac{n-2}{2})$.
 
\item[(c)] For $0=\Theta\geqslant \Phi$, $ 0<\nu_{\Phi}<\frac{n-1}{2}+\Phi$, $\mathfrak l_0\cong \mathfrak u(n-1,1)\oplus\mathfrak u(0,1) \oplus\mathfrak u(1,0)$, and 
\[\lambda=(\underbrace{-1,-1,\cdots,-1}_{n};\Phi+\nu_\Phi+\frac{n-1}{2};\Phi-\nu_\Phi+\frac{n+1}{2}).\]

For $\Theta\geqslant 0=\Phi$, $0<\nu_{\Theta}<\frac{n-1}{2}-\Theta$, $\mathfrak l_0\cong \mathfrak u(1,0) \oplus\mathfrak u(0,1) \oplus\mathfrak u(n-1,1)$, and 
\[\lambda=(\Theta+\nu_\Theta-\frac{n+1}{2};\Theta-\nu_\Theta-\frac{n-1}{2};\underbrace{1,1,\cdots,1}_{n}).\] 
We only consider the case  $0=\Theta\geqslant \Phi$: The spin lowest $K$ type is: \[(1,0,\dots,0,-1\ |\ \Phi-\nu_{\Phi}+\frac{n-1}{2},\Phi+\nu_{\Phi}-\frac{n-1}{2})\]
and the corresponding $\rho_n^{(j)}=(1,0,\dots,0,-1\ |\ \frac{n-2}{2},\frac{-n+2}{2})$. 

\item[(d)] For $\Theta\geqslant\Phi>0$,  $\mathfrak l_0\cong \mathfrak u(n-1,2) \oplus \mathfrak u (1,0)$ and 
\[\lambda=(\underbrace{0,0,\cdots,0}_{n+1}; 2\Theta+1).\]
For $0>\Theta\geqslant\Phi$, $\mathfrak l_0\cong \mathfrak u(1,0) \oplus \mathfrak u (n-1,2)$ and 
\[\lambda=(2\Phi-1;\underbrace{0,0,\cdots,0}_{n+1}).\] 
For $\Theta=\Phi=0$, it corresponds to the trivial module. 

We only consider the case $\Theta\geqslant\Phi>0$:
the spin lowest $K$-types are
\[(0,\dots,0\ |\ 2\Theta,1), \  \text{and}\ (0,\dots,0,-1\ |\ 2\Theta+1,1).\]
The first corresponds to 
$\rho_n^{(j)}=(\underbrace{0,\dots,0}_{s},-1,\dots,-1\ |\ \frac{n}{2},\frac{n-2s}{2})$, $0\leqslant s\leqslant 2\Theta-1$,
and the second corresponds to $\rho_n^{(j)}=(\underbrace{0,\dots,0}_{s},-1,\dots,-1\ |\ \frac{n}{2},\frac{n-2s}{2})$,  $2\Theta \leqslant s\leqslant n-1$. 
\end{itemize}
By the above results of the spin lowest $K$-types and the corresponding $\rho_n^{(j)}$, the theorem can be checked easily. Let us give the sketch of the proof.

In the case (a), (b), the computation of Dirac cohomology follows from the claim below.

Claim: Assume that $L=U(p,1)\times U(n-p,1)$, and $\lambda=(-1,\dots,-1;1,\dots,1)$, then the Dirac cohomology has unique $\widetilde{K}$-type with multiplicity two (resp. one) when $p, n-p>0$ (resp. $p=0$ or $n-p=0$), and there are two (resp. one) spin lowest $K$-type with multiplicity one. 

Let us show the claim:
When $p,n-p>0$, one has 
\[\lambda+\rho=(p,\dots,1,0;1,0,\dots,-n+p+1)+(\frac{n-2p-1}{2},\dots,\frac{n-2p-1}{2}).\]
Notice that $1,0$ appears twice in $\lambda+\rho$, and there exists only two $w\in W^1$ such that $w(\lambda+\rho)$ is $\Delta^+(\mathfrak{k},\mathfrak{t})$-regular, and they give the same $\{w(\lambda+\rho)\}$. So there is only one possible $\widetilde{K}$-type in Dirac cohomology:
\[\{w(\lambda+\rho)\}=(p,\dots,-n+p+1\ |\ 1,0)-\rho_c+(\frac{n-2p-1}{2},\dots,\frac{n-2p-1}{2}).\]
By the Blattner formula, with a similar argument as 
\cite[Proposition 5.3]{WongZhang24}, one can see that all the $K$-types of $A_{\mathfrak{q}}(\lambda)$ are of the form
\[(0,\dots,0\ |\ n-p-1,-p+1)+(a,0,\dots,0,-b\ |\ b,-a);\ a,b\in \mathbb{N}\]
with multiplicity one. A little computation  finds that there are two spin lowest $K$-types
\[a=1,b=0\ \text{or}\ a=0,b=1.\]
where the first one corresponds to $\rho_n^{(j)}=(1,0,\dots,0\ | \ \frac{p}{2}-1,-\frac{p}{2})$, and the second one corresponds to 
 $\rho_n^{(j)}=(0,\dots,0,-1\ | \ \frac{p}{2},-\frac{p}{2}+1)$,

When $p=0$, $n-p>1$, By the Blattner formula, all the $K$-types are of the form
\[(0,\dots,0,-s\ | \ s+n-1,1),\ s\in \mathbb{Z},\ s\geqslant 0,\]
and have multiplicity one.

One can find that here are spin lowest $K$-types 
\[(0,\dots,0,-1\ |\ n,1),\ \text{and}\ (0,\dots,0\ |\ n-1,1).\]
The first corresponds to $\rho_n^{(j)}=(0,\dots,0,-1\ | \ \frac{n}{2},-\frac{n-2}{2})$ and the second corresponds to
$\rho_n^{(j)}=(0,\dots,0,\underbrace{-1,\dots,-1}_{m+1}\ | \  \frac{n}{2},-\frac{n-2-2m}{2}),\ 1\leqslant m\leqslant n-1$. Now the claim follows.
 
In the case (c), assume that $w(\lambda+\rho)-\rho_c$ shows up in the Dirac cohomology. since $\lambda+\rho$ (which is the infinitesimal character of $A_{\mathfrak{q}}(\lambda)$) has two pair of the coordinates which are the same, there exists exactly one element  
$\{w(\lambda+\rho)\}$, $w\in W(\mathfrak{g},\mathfrak{t})$, which is $\Delta^+(\mathfrak{k},\mathfrak{t})$-regular and dominant. Moreover, $\{w(\lambda+\rho)\}-\rho_c=(0,\dots,0\ |\ \Phi+\nu_{\Phi},\Phi-\nu_{\Phi})$.
By the Blattner formula and \eqref{spin_lowest}, a little argument shows that the possible spin lowest K-types of $A_{\mathfrak{q}}(\lambda)$ are of the form
\[(a,0,\dots,0,\Phi-\nu_{\Phi}+\frac{n-3}{2}-b-c\ |\ b,\Phi+\nu_{\Phi}-\frac{n-3}{2}+c-a),\ a,b,c\in\mathbb{N}. \]
Now, one can easily determine the spin lowest $K$ type: \[(1,0,\dots,0,-1\ |\ \Phi-\nu_{\Phi}+\frac{n-1}{2},\Phi+\nu_{\Phi}-\frac{n-1}{2})\]
and the corresponding $\rho_n^{(j)}=(1,0,\dots,0,-1\ |\ \frac{n-2}{2},\frac{-n+2}{2})$.

In the case (d), by $w(\lambda+\rho)$ is $\mathfrak{k}$-regular, and $3\leqslant 2\Theta+1\leqslant n-1$, it is easy to see that all the $\mathfrak{k}$-regular element of $\{w(\lambda+\rho)\}$ are
\[\begin{aligned}& \rho_c+\big(\underbrace{1,\dots,1}_{r},\underbrace{0,\dots,0}_{n-r}\ |\ -\frac{n}{2}+2\Theta,\frac{n}{2}-r+1\big),\ \ n-2\Theta+1\leqslant r\leqslant n; \ \text{and}\\
&\rho_c+\big(\underbrace{1,\dots,1}_{r},\underbrace{0,\dots,0}_{n-r}\ |\ \frac{n}{2}-r, -\frac{n}{2}+2\Theta+1\big),\ \ 0\leqslant r\leqslant n-2\Theta-1.\end{aligned}\]
These are all the $K$-types which are possible to show up in the Dirac cohomology.

Let us show that they all  show up in the Dirac cohomology. By Blattner formula, a little bit of argument shows that
all the $K$-types of $A_{\mathfrak{q}}(\lambda)$ are of the form 
\[(0,\dots,0, 2\Theta-1-m\ |\ m+1,1),\ m\in \mathbb{Z},\ m\geqslant 2\Theta-1,\]
and all the $K$-types have multiplicity one.
Since the spin lowest $K$-type $\sigma$ satisfies $\{w(\lambda+\rho)\}-\rho_c=\{\sigma-\rho_n^{(j)}\}$ for some $\rho_n^{(j)}$.
Notice that $\rho_n^{(j)}=(1,\dots,1,0,\dots,0,-1,\dots,-1\ |\ *,*)$, since $\{w(\lambda+\rho)\}-\rho_c=\{\sigma-\rho_n^{(j)}\}$, one must have 
\[\rho_n^{(j)}=(\underbrace{0,\dots,0}_{s},-1,\dots,-1\ |\ \frac{n}{2},\frac{n-2s}{2}),\ 0\leqslant s\leqslant n-1.\] 
Now one can get the corresponding spin lowest $K$-types are
\[\begin{cases}  (0,\dots,0\ |\ 2\Theta,1),\ & \text{when}\ 0\leqslant s\leqslant 2\Theta-1\  \\ (0,\dots,0,-1\ |\ 2\Theta+1,1),\  & \text{when}\ 2\Theta \leqslant s\leqslant n-1.\end{cases}\]
\end{proof}

\end{document}